\documentclass[11pt, reqno, en, oneside]{amsart}%
\usepackage[font=scriptsize]{caption}
\usepackage[ocgcolorlinks]{hyperref}
\usepackage{aliascnt}
\newcommand{\mytheorem}[3]{%
  \newaliascnt{#1}{#3}
  \newtheorem{#1}[#1]{#2}
  \aliascntresetthe{#1}
  \expandafter\newcommand\csname #1autorefname\endcsname{#2}
}
\usepackage{mathrsfs,amssymb,amsthm, amsmath, mathtools}
\usepackage[%
a4paper,
margin=1in,
marginpar=.8in,
includeheadfoot%
]{geometry}
\usepackage{graphicx}
\usepackage{etoolbox}
\usepackage{ifxetex}
\usepackage{enumerate}
\usepackage{enumitem}
\makeatletter
\renewcommand*\env@cases[1][1.15]{%
  \let\@ifnextchar\new@ifnextchar
  \left\lbrace
    \def\arraystretch{#1}%
    \array{@{}l@{\quad}l@{}}%
  }
  \makeatother
  \usepackage[dvipsnames]{xcolor}
  \definecolor{darkgreen}{rgb}{0,0.5,0}
  \definecolor{darkred}{rgb}{0.7,0,0}

  \providetoggle{langen}
  \makeatletter
  \newif\ifuserdefn@presubmit
  \ifuserdefn@presubmit
    \usepackage[left, modulo, displaymath, mathlines]{lineno}
    \linenumbers
    \newcommand*\patchAmsMathEnvironmentForLineno[1]{%
      \expandafter\let\csname old#1\expandafter\endcsname\csname #1\endcsname
      \expandafter\let\csname oldend#1\expandafter\endcsname\csname end#1\endcsname
      \renewenvironment{#1}%
      {\linenomath\csname old#1\endcsname}%
    {\csname oldend#1\endcsname\endlinenomath}}%
    \newcommand*\patchBothAmsMathEnvironmentsForLineno[1]{%
      \patchAmsMathEnvironmentForLineno{#1}%
    \patchAmsMathEnvironmentForLineno{#1*}}%
    \AtBeginDocument{%
      \patchBothAmsMathEnvironmentsForLineno{equation}%
      \patchBothAmsMathEnvironmentsForLineno{align}%
      \patchBothAmsMathEnvironmentsForLineno{flalign}%
      \patchBothAmsMathEnvironmentsForLineno{alignat}%
      \patchBothAmsMathEnvironmentsForLineno{gather}%
      \patchBothAmsMathEnvironmentsForLineno{multline}%
      \patchAmsMathEnvironmentForLineno{multlined}
    }
    \usepackage{seqsplit}
    \usepackage[right]{showlabels}
    \usepackage{xstring}
    \renewcommand{\showlabelsetlabel}[1]{
      \noexpandarg%
      \parbox[t]{\marginparwidth}{\raggedright\upshape\small\color{blue}\expandafter\seqsplit\expandafter#1}
    }
    \AtBeginDocument{\newgeometry{
        vmargin={0in,0in},
        hmargin={0.8in,1.3in},
        includeheadfoot
      }
    }
  \fi
  \makeatother

  \iftoggle{langen}{
      \usepackage[british]{babel}
      \usepackage{lmodern}
      \usepackage[T1]{fontenc}
      \newtheoremstyle{scheader}%
      {}% space above, empty=default of current
      {}% space below
      {\itshape}% body font
      {}% indent amount, empty=no indent, \parindent = normal paragraph indent
      {\scshape}% theorem head font
      {.}% punctuation after theorem head
      {0.5em}% space after theorem head, empty=normal interward space, \newline =line break
      {}% theorem head spec, empty=normal
      \theoremstyle{definition}
      \newtheorem{defn}{Definition}[section]
      \theoremstyle{scheader}
      \newtheorem{mthm}{Theorem}

      \mytheorem{mcor}{Corollary}{mthm}
      \mytheorem{mlem}{Lemma}{mthm}

      \mytheorem{thm}{Theorem}{defn}
      \mytheorem{lem}{Lemma}{defn}
      \mytheorem{cor}{Corollary}{defn}
      \mytheorem{prop}{Proposition}{defn}
      
      \theoremstyle{remark}
      \newtheorem*{rmk}{Remark}
      \newtheorem{nrmk}{Remark}
      \newtheorem*{claim}{Claim}
      \newtheorem{nclaim}{Claim} %number claim
      
      \newcommand{\I}{\mathrm{I}}
      }{
      \newtheorem{thm}{定理}[section]
      \mytheorem{defn}{定义}{thm}
      \mytheorem{lem}{引理}{thm}
      \mytheorem{cor}{推论}{thm}
      \mytheorem{prop}{命题}{thm}
      \newtheorem*{rmk}{注记}

      \newtheorem*{claim}{断言}

      \renewcommand{\emph}[1]{\textbf{\texttt{#1}}}
      \newcommand{\I}{\mathrm{I}}
    }
    \AtBeginEnvironment{proof}{%
      \setcounter{stepnum}{0}
      \setcounter{nclaim}{0}
    }
    \newcounter{stepnum}
    \newcommand{\step}{%
      \par%
      \refstepcounter{stepnum}%
      {\scshape Step~\arabic{stepnum}}.%
      \enspace\ignorespaces
    }
    \newcommand{\eps}{\varepsilon}
    \newcommand{\loc}{\mathrm{loc}}

    \def\XXint#1#2#3{{\setbox0=\hbox{$#1{#2#3}{\int}$ }
    \vcenter{\hbox{$#2#3$ }}\kern-.6\wd0}}

    \newcommand{\weakto}{\rightharpoonup}
    \newcommand{\xtext}[1]{\ensuremath{\,\text{ #1 }}}
    \DeclareRobustCommand{\set}[1]{\left\{#1\right\}}

    \newcommand{\II}{\mathrm{II}}
    \renewcommand{\i}{\rm{i}}

    \newcommand{\myquad}[1][1]{\hspace*{#1em}\ignorespaces}

    \usepackage[backrefs, msc-links, abbrev]{amsrefs} %, alphabetic
    \hypersetup{%
      pdfstartview=FitH,
      bookmarksopen=false,
      pdfencoding=auto,
      colorlinks=false,
      linkcolor=blue,
      pdfinfo={%
        Subject={58E15, 35J50, 35R35},
        Keywords={Yang--Mills--Higgs, free boundary, blow-up}%
      }
    }
    \usepackage[perpage]{footmisc}
    
    \numberwithin{equation}{section}
    \renewcommand{\PrintDOI}[1]{%
      DOI {\href{https://doi.org/#1}{#1}}%
      \IfEmptyBibField{volume}{, (to appear in print)}{}%
    }
    
    \newcommand{\lh}{\,\rule[-.25em]{.5em}{.05em}\rule[-.25em]{.05em}{1em}\,}
    
    \newcommand{\id}{\operatorname{Id}}
    \newcommand{\Span}[1]{\operatorname{span}\left\{ #1 \right\}}
    \makeatletter
    \newcommand{\multi}[1]{#1\checknextchar}
    \newcommand{\checknextchar}{\@ifnextchar\bgroup{\raisebox{0.4\height}{\scalebox{0.62}{$\#$}}\multi}{}}
    \makeatother

    \title[Boundary value problem for Yang--Mills--Higgs]{The boundary value problem for Yang--Mills--Higgs fields}
    \author{Wanjun Ai}
    \address{School of Mathematics and Statistics,
      Southwest University,
      Chongqing, 400715, P. R. China\\
      \indent School of Mathematical Sciences, %
      Shanghai Jiao Tong University, %
    Shanghai 200240, P. R. China}%
    \email{wanjunai@swu.edu.cn; aiwanjun@sjtu.edu.cn}%
    \author{Chong Song}
    \address{School of Mathematical Sciences, %
      Xiamen University, %
    Xiamen 361005, P. R. China}%
    \email{songchong@xmu.edu.cn}
    \author{Miaomiao Zhu}
    \address{School of Mathematical Sciences, %
      Shanghai Jiao Tong University, %
    Shanghai 200240, P. R. China}%
    \email{mizhu@sjtu.edu.cn}%

    \thanks{Part of this work was carried out when Wanjun Ai was a postdoc at the School of Mathematical Sciences,
      Shanghai Jiao Tong University and he would like to thank the institution for hospitality and financial support. Chong Song is partially supported by the Fundamental Research Funds for the Central Universities (Grant No. 20720170009, 20720180009). Miaomiao Zhu was supported in part by National Natural Science Foundation of China (No. 11601325). We would like to thank the referee for careful comments and helpful suggestions in improving the presentation of the paper.}
    \date{\today}

\begin{document}
    %\allowdisplaybreaks[4]
    \begin{abstract}
      We show the existence of Yang--Mills--Higgs (YMH) fields over a Riemann surface with boundary where a free boundary condition is imposed on the section and a Neumann boundary condition on the connection. In technical terms, we study the convergence and blow-up behavior of a sequence of Sacks--Uhlenbeck type $\alpha$-YMH fields as $\alpha\to 1$. For $\alpha>1$, some regularity results for $\alpha$-YMH field are shown. This is achieved by showing a regularity theorem for more general coupled systems, which extends the classical results of Ladyzhenskaya--Ural\cprime ceva and Morrey.
    \end{abstract}
    \maketitle
    \section{Introduction}\label{sec:intro}
    The Yang--Mills--Higgs (YMH) theory arises from the research of electromagnetic phenomena and plays a fundamental role in modern physics, especially in quantum field theories. Due to its remarkable applications in both geometry and topology, the YMH theory has been extensively studied by mathematicians in the last several decades.

    The general YMH theory can be modeled in the following setting. Suppose $\Sigma$ is a Riemannian manifold, $G$ is a compact Lie group with Lie algebra $\mathfrak{g}$, which is endowed with a left-invariant metric, and $\mathcal{P}$ is a $G$-principal bundle on $\Sigma$. Let $F$ be a Riemannian manifold admitting a $G$-action, and $\mathcal{F}=\mathcal{P}\times_G F$ be the associated fiber bundle. Suppose there is a generalized Higgs potential $\mu$ which is a smooth gauge invariant vector-valued function on $\mathcal{F}$. Let $\mathscr{S}$ denote the space of smooth sections of $\mathcal{F}$, and $\mathscr{A}$ denote the affine space of smooth connections on $\mathcal{P}$. Then the YMH functional is defined for a pair $(A,\phi)\in \mathscr{A}\times\mathscr{S}$ by
    \[
      \mathcal{L}(A,\phi)\mathpunct{:}=\lVert \nabla_A\phi\rVert_{L^2}^2+\lVert F_A\rVert_{L^2}^2+\lVert \mu(\phi)\rVert_{L^2}^2,
    \]
    where $F_A$ is the curvature of $A$, $\nabla_A$ is the covariant differential corresponding to $A$. The exterior extension of $\nabla_A$ is denoted by $D_A$, i.e., the exterior covariant differential. Critical points of the above YMH functional $\mathcal{L}$ are called YMH fields, which satisfy the following Euler--Lagrange equation on $\Sigma$:
    \begin{equation}\label{eq:EL}
      \begin{cases}
        \nabla_A^*\nabla_A\phi+\mu(\phi)\cdot\nabla\mu(\phi)=0,\\
        D_A^*F_A+\left\langle \nabla_A\phi,\phi \right\rangle=0.
      \end{cases}
    \end{equation}

    As the Lie groups and the manifolds differ, the above YMH framework covers many variants. For example, if $F$ is a point, then the YMH theory reduces to the usual Yang--Mills theory. If the Lie group $G$ is trivial, then the YMH fields are just harmonic maps (with potential). When $\mathcal{F}$ is a complex line bundle and $G=S^1$ is abelian, then we recover the classical Ginzburg--Landau equations in the theory of superconductivity \citelist{\cite{BethuelBrezisHelein1994Ginzburg}\cite{JaffeTaubes1980Vortices}}. In general, $G$ can be any compact and possibly non-abelian Lie group and the YMH model can also be viewed as a version of gauged sigma model  with certain supersymmetries determined by the type of $F$ ~\citelist{\cite{Alvarez-GaumeFreedman1981Geometrical}\cite{BaggerWitten1982gauge}\cite{Witten1993Phases}}. A particular important case is when both $\Sigma$ and $F$ are K\"ahler manifolds and there is a holomorphic structure on $\mathcal{F}$, then the minimal points of the YMH functional satisfies a first-order equation and are usually referred to as vortices. The existence of vortices has deep relations with the notion of stability conditions, which is now known as the Hitchin--Kobayashi correspondence~(see for example \citelist{\cite{Donaldson1985anti}\cite{UhlenbeckYau1986existence}\cite{Bradlow1991Special}\cite{Banfield2000Stable}\cite{Riera2000HitchinKobayashi}}). Moreover, the moduli space of vortices can be used to construct symplectic invariants of $F$ with respect to the group action, which is a generalization of the celebrated Gromov--Witten invariants~\citelist{\cite{CieliebakGaioSalamon2000holomorphic}\cite{Riera2003Hamiltonian}}.
    On the other hand, non-minimal YMH fields do exists, see for example~\citelist{\cite{Taubes1982existenceI}\cite{Taubes1982existenceII}}.

    From now on, we assume $\Sigma$ is a compact Riemann surface with non-empty boundary $\partial \Sigma$, $F$ is a compact Riemannian manifold and $G$ is a connected compact Lie group. We will investigate the existence of general YMH fields satisfying the second-order Euler--Lagrange equations~\eqref{eq:EL} under appropriate boundary conditions, namely, a free boundary condition imposed on the section and a Neumann boundary condition on the connection.

    The existence of general YMH fields on a \emph{closed} Riemann surface has been studied by Song~\cite{Song2011CriticalYangMillsHiggs}. The corresponding gradient flow of the YMH functional is investigated by Yu~\cite{Yu2014gradient} under the name of gauged harmonic maps following the work of Lin--Yang~\cite{LinYang2003Gauged} and by Song--Wang~\cite{SongWang2017Heat}. Song~\cite{Song2016Convergence} also studied the convergence of YMH fields where the conformal structure of the underlying surface $\Sigma$ is allowed to vary and degenerate. When $\Sigma$ has possibly non-empty boundary, the minimal YMH fields in the holomorphic setting are studied by Xu~\cite{Xu2013moduli} and Venugopalan~\cite{Venugopalan2016YangMills}. See also the very recent paper by Lin--Shen~\cite{LinShen2018Gradient} on the heat flow of the YMH functional over a compact K\"ahler manifold.

    A closely related problem is harmonic maps from surfaces, which has been extensively studied. For example, the convergence of harmonic maps from degenerating surfaces was firstly systematically explored in \cite{Zhu2010Harmonic} and the existence of harmonic maps with free boundary was studied via various approaches~\citelist{\cite{Fraser2000free}\cite{GulliverJost1987Harmonic}\cite{Ma1991Harmonic}}. On one hand, since the Dirichlet energy $\lVert \nabla_A\phi\rVert_{L^2}^2$ is critical in dimension two, we shall follow the general scheme developed for two-dimensional harmonic map type problems to deal with the section part $\phi$. On the other hand, although the Yang--Mills energy $\lVert F_A\rVert_{L^2}^2$ is subcritical in dimension two,  however, as we will see in this paper, the coupled system brings new technical difficulties caused by the connection part $A$. One of the main achievements in the present paper is to overcome them (see the remark after \autoref{thm:Morrey} for more details).

    Now we shall describe our boundary value problem for YMH fields $(A,\phi)$ in more precise terms. Let $K \subset F$ be a closed sub-manifold which is invariant under the $G$-action. Let $\mathcal{K}=\mathcal{P}\times_GK$ be the sub-bundle of $\mathcal{F}$ with fiber $K$, define the space of smooth sections of $\mathcal{F}$ with free boundary as
    \[
      \mathscr{S}_K \mathpunct{:}=\left\{ \phi\in \mathscr{S}:\phi|_{\partial\Sigma}\in \mathcal{K} \right\}.
    \]
    Clearly, the tangent space of $\mathscr{S}_K$ at $\phi$ is given by
    \[
      T_\phi \mathscr{S}_K=\left\{ \psi\in\Gamma(\phi^{*}T\mathcal{F}^v):\psi(x)\in T_{\phi(x)}\mathcal{K}^v,\,x\in\partial\Sigma \right\},
    \]
    where $T \mathcal{F}^v$ denotes the vertical distribution of tangent bundle $T \mathcal{F}$. On the other hand, the affine space of connections of principal bundle over $\Sigma$ with $\partial\Sigma\neq\emptyset$ is still denoted by $\mathscr{A}$. The tangent space of $\mathscr{A}$ at $A$ is $T_A \mathscr{A}=\Omega^1(\mathfrak{g}_{\mathcal{P}})$, where $\mathfrak{g}_{\mathcal{P}}\mathpunct{:}=\mathcal{P}\times_{\mathrm{Ad}}\mathfrak{g}$ is the Lie algebra vector bundle.
    A simple computation yields the first variation of $\mathcal{L}$ on $\mathscr{A}\times \mathscr{S}_K$,
    \begin{multline*}
      \delta_{\xi,\psi}(\mathcal{L}(A,\phi))=2\int_{\Sigma}\left\langle \nabla_A^*\nabla_A\phi,\psi \right\rangle+\left\langle \mu(\phi)\cdot\nabla\mu(\phi),\psi \right\rangle+\left\langle D_A^*F_A,\xi \right\rangle+\left\langle \nabla_A\phi,\xi\phi \right\rangle\\
      +2\int_{\partial\Sigma}\left\langle \psi,\nu\lh \nabla_A\phi \right\rangle+\left\langle \xi,\nu\lh F_A \right\rangle,
    \end{multline*}
    where $\xi\in T_A\mathscr{A}$ and $\psi\in T_\phi \mathscr{S}_K$, $\nu$ is the unit outer normal vector filed on $\partial\Sigma$ and $\nu\lh$ denotes the contraction of a form with $\nu$. Therefore, a critical point $(A, \phi) \in \mathscr{A}\times \mathscr{S}_K$ of $\mathcal{L}$ satisfies the Euler--Lagrangian equation~\eqref{eq:EL} in the interior of $\Sigma$ and satisfies the following boundary condition on $\partial\Sigma$,
    \[
      (\hypertarget{bdry:N}{\mathcal{N}})\mathpunct{:}\quad
      \begin{cases}
        \nu\lh  \nabla_A\phi\perp T_\phi\mathcal{K}^v,\\
        \nu\lh F_A=0.
      \end{cases}
    \]

    \begin{defn}
      A smooth pair $(A,\phi)\in\mathscr{A}\times\mathscr{S}_K$ is called a \emph{YMH field} with free boundary on the section and Neumann boundary on the connection if it satisfies the system~\eqref{eq:EL} in the interior of $\Sigma$ and satisfies the boundary condition (\hyperlink{bdry:N}{$\mathcal{N}$}) on the boundary $\partial\Sigma$.
    \end{defn}

    Our main goal in this paper is to show the existence of such YMH fields on $\Sigma$. Note that in dimension two, the above condition for the connection $A$ simply means $*F_A=0$ on $\partial\Sigma$, which is exactly the Neumann boundary condition in the study of Yang--Mills theory (see e.g.~\cite{Marini1992Dirichlet}). On the other hand, if we take $\mathcal{K}=\mathcal{F}$ to be the total bundle, then the boundary condition for the section $\phi$ reduces to the Neumann boundary condition $\langle \nu, \nabla_A\phi\rangle =0$ on $\partial \Sigma$. In the case of Ginzburg--Landau theory with $\mathcal{K}=\mathcal{F}$, this boundary condition (\hyperlink{bdry:N}{$\mathcal{N}$}) coincides with the natural \emph{homogeneous} de Gennes--Neumann boundary condition in the study of superconductivity (see e.g. \citelist{\cite{DeGennes1966superconductivity}\cite{LuPan1996GinzburgLandau}} for non-homogeneous condition of sections, and \citelist{\cite{ChapmanHowisonOckendon1992Macroscopic}\cite{DuGunzburgerPeterson1992nalysis}\cite{Nagy2017Irreducible}} for the homogeneous one).

    To investigate the existence of YMH fields subject to the boundary condition (\hyperlink{bdry:N}{$\mathcal{N}$}), in contrast to the Ginzburg--Landau case (see \cite{Nagy2017Irreducible}*{Lem.~3.1}), $\mathcal{L}$ does not satisfy the Palais--Smale condition anymore and we follow the scheme of~\citelist{\cite{SacksUhlenbeck1981existence}\cite{Song2011CriticalYangMillsHiggs}} by considering the following perturbed $\alpha$-functional for  $\alpha>1$:
    \[
      \mathcal{L}_\alpha(A,\phi)\mathpunct{:}=\int_\Sigma(1+|\nabla_A\phi|^2)^\alpha+\lVert F_A\rVert_{L^2}^2+\lVert \mu(\phi)\rVert_{L^2}^2,\quad (A,\phi)\in \mathscr{A}_1^2\times \mathscr{S}_{1,K}^{2\alpha},
    \]
    where $\mathscr{A}_1^2$ and $\mathscr{S}_{1,K}^{2\alpha}$ denote the corresponding Sobolev spaces which are defined as follows: for a fixed smooth connection $A_0\in \mathscr{A}$, the \emph{affine} Sobolev space of $L_1^p$ connections is defined as
    \[
      \mathscr{A}_1^p \mathpunct{:}=\left\{ A\in A_0+L_1^p\left( \Omega^1\bigl(\mathfrak{g}_{\mathcal{P}}\bigr) \right) \right\}.
    \]
    The spaces $\mathscr{A}_1^p$ defined via different choices of $A_0$ are isomorphic to each other. The Sobolev space of sections $\mathscr{S}_{1,K}^{2\alpha}$ is defined by
    \[
      \mathscr{S}_{1,K}^{2\alpha}\mathpunct{:}=\left\{ \phi\in L_1^{2\alpha}(\mathcal{E}):\phi(x)\in \mathcal{F}\xtext{for a.e. $x\in \Sigma$}\xtext{and}\phi(x)\in \mathcal{K}\xtext{for a.e. $x\in\partial\Sigma$} \right\},
    \]
    where we embed $\mathcal{F}$ into a vector bundle $\mathcal{E}=\mathcal{P}\times_G \mathbb{R}^l$ for some large enough $l$ such that $F \hookrightarrow \mathbb{R}^l$ is an \emph{equivariant} (with respect to the orthogonal representation $\rho \mathpunct{:}G\to \mathrm{O}(l)$) isometrical embedding (see \cite{MooreSchlafly1980equivariant}*{Main Thm.}), and we view sections of $\mathcal{F}$ as sections of $\mathcal{E}$, where the covariant differential induced by $A$ is also defined. We refer to~\cite{Wehrheim2004Uhlenbeck}*{Appx.~B} for the definition of Sobolev norms on vector bundles and fiber bundles (e.g., the gauge group of $\mathcal{P}$, $\mathscr{G}_2^p \mathpunct{:}= L_2^p(\mathcal{P}\times_c G)$, where $c$ is the conjugation).

    It turns out that, the perturbed functional $\mathcal{L}_\alpha$ with $\alpha>1$ satisfies the Palais--Smale condition on $\mathscr{A}_1^2\times \mathscr{S}_{1,K}^{2\alpha}$  (see~\autoref{sec:PS}), hence it admits critical points, which we call \emph{$\alpha$-YMH fields} with free boundary on the section and Neumann boundary on the connection, by classical theory of calculus of variation. The Euler--Lagrange equation for a critical point $(A, \phi)$ of $\mathcal{L}_\alpha$ is given by
    \begin{equation}\label{eq:EL-L-alpha-global}
      \begin{cases}
        \nabla_A^*\left( \alpha(1+|\nabla_A\phi|^2)^{\alpha-1}\nabla_A\phi \right)-\mu(\phi)\cdot\nabla\mu(\phi)=0,&x\in\Sigma\\
        D_A^*F_A+\alpha(1+|\nabla_A\phi|^2)^{\alpha-1}\left\langle \nabla_A\phi,\phi \right\rangle=0,&x\in\Sigma\\
        \nu\lh F_A=0,&x\in\partial\Sigma\\
        \nu\lh \nabla_A\phi\perp T_\phi\mathcal{K}^v,&x\in\partial\Sigma.
      \end{cases}
    \end{equation}

    Our first result is the following interior regularity theorem and boundary regularity theorem for $\alpha$-YMH fields under our boundary condition (\hyperlink{bdry:N}{$\mathcal{N}$}).
    \begin{mthm}\label{thm:alpha-smoothness}
      Suppose $\alpha>1$, $K\subset F$ is a $G$-invariant sub-manifold and $(A_\alpha,\phi_\alpha)$ is a critical point of $\mathcal{L}_\alpha$ in $\mathscr{A}_1^2\times \mathscr{S}_{1,K}^{2\alpha}$. Then for any compact subset $\Sigma'$ in the interior of $\Sigma$, there exists a gauge transformation $\tilde{S}\in \mathscr{G}_2^2(\Sigma')$, such that $(\tilde{S}^*A_\alpha, \tilde{S}^*\phi_\alpha)$ is smooth on $\Sigma'$. If, in addition, $K\subset F$ is a totally geodesic sub-manifold, then there exists a gauge transformation $\tilde{S}\in \mathscr{G}_2^2(\Sigma)$, such that $(\tilde{S}^*A_\alpha,\tilde{S}^*\phi_\alpha)$ is smooth up to the boundary.
    \end{mthm}
    For $\alpha$-harmonic maps, which can be regarded as a special kind of $\alpha$-YMH fields, such regularity result was proved by Sacks-Uhlenbeck~\cite{SacksUhlenbeck1981existence}*{Prop.~2.3} in the case of a closed domain and the free boundary case was considered by Fraser~\cite{Fraser2000free}*{Prop.~1.4}. The proof for the case of $\alpha$-harmonic maps simply follows from a classic regularity theorem by Morrey~\cite{Morrey2008Multiple}*{Thm.~1.11.1$'$, p.~36}, extending the one by Ladyzhenskaya-Ural\cprime ceva ~\cite{LadyzhenskayaUraltseva1968Linear}*{Chap.~8, Thm.~2.1, p.~412}. However, Morrey's theorem in~\cite{Morrey2008Multiple}*{Thm.~1.11.1$'$} can not be applied to the coupled system of $\alpha$-YMH field $(A, \phi)\in \mathscr{A}_1^2\times \mathscr{S}_{1,K}^{2\alpha}$, because the corresponding ellipticity condition in (1.10.8$''$) of~\cite{Morrey2008Multiple}*{Thm.~1.11.1$'$} cannot be verified, due to the feature of the non-trivial coupling between the two fields. Therefore, we need to develop a new regularity theorem to handle coupled systems of more general type, in particular, to include the system of $\alpha$-YMH fields\footnote{This technical issue was overlooked by Song in \cite{Song2011CriticalYangMillsHiggs} and here we take our opportunity to fix the gap by extending Morrey's theorem to Theorem B.}.

    In this paper, we succeed in deriving such a more general regularity result, which itself is interesting and might lead to applications to various other coupled systems emerging from geometry and physics.

    Suppose $\Omega\subset \mathbb{R}^n$ is a bounded domain, we will consider a 2-coupled system,
    \begin{equation}\label{eq:2coupled-system}
      \begin{dcases}
        -\sum_{\alpha=1}^n\partial_\alpha q_{1i}^\alpha(x,z,\nabla z)+w_{1i}(x,z,\nabla z)=0,& i=1,2,\ldots, m_1,\\
        -\sum_{\alpha=1}^n\partial_\alpha q_{2i}^\alpha(x,z,\nabla z)+w_{2i}(x,z,\nabla z)=0,& i=1,2,\ldots, m_2,
      \end{dcases}
    \end{equation}
    where $x=(x^1,\ldots,x^n)\in\Omega$, $\partial_\alpha=\partial_{x^\alpha}$; $z=z(x)=(z_1(x),z_2(x))\in \mathbb{R}^{m_1}\times \mathbb{R}^{m_2}$, and $\nabla z=(\nabla z_1,\nabla z_2)\in \mathbb{R}^{m_1n}\times \mathbb{R}^{m_2n}$ is the gradient of $z$; $q=(q_1,q_2)$, $q_a=(q_{ai}^\alpha)_{m_a\times n}$; $w=(w_1,w_2)$, $w_a=(w_{ai})_{m_a\times 1}$. A vector valued function $z(x)=(z_1(x), z_2(x))\in L_1^{k_1}(\Omega, \mathbb{R}^{m_1})\times L_1^{k_2}(\Omega, \mathbb{R}^{m_2})$, $k_1\geq 2$, $k_2\geq 2$, is called a \emph{weak} solution of \eqref{eq:2coupled-system} if
    \begin{equation}\label{eq:coupled-weak}
      \begin{dcases}
        \int_\Omega  \sum_{i=1}^{m_1}\left( \sum_{\alpha=1}^n\partial_\alpha \xi_1^i(x)q_{1i}^\alpha(x,z,\nabla z) +\xi_1^i(x)w_{1i}(x,z,\nabla z) \right) =0,&\forall\xi_1\in L_{1,0}^{k_1}\bigcap C^0(\Omega, \mathbb{R}^{m_1}),\\
        \int_\Omega  \sum_{i=1}^{m_2}\left( \sum_{\alpha=1}^n\partial_\alpha \xi_2^i(x)q_{2i}^\alpha(x,z,\nabla z) +\xi_2^i(x)w_{2i}(x,z,\nabla z) \right) =0,&\forall\xi_2\in L_{1,0}^{k_2}\bigcap C^0(\Omega, \mathbb{R}^{m_2}).
      \end{dcases}
    \end{equation}
    We will assume that the coefficients $q=q(x,z,p)\in L_{1,\loc}^1\bigcap C^0(\Omega\times \mathbb{R}^{m}\times \mathbb{R}^{mn})$ and $w=w(x,z,p)\in L_{1,\loc}^1\bigcap C^0(\Omega\times \mathbb{R}^{m}\times \mathbb{R}^{mn})$, $m=m_1+m_2$, satisfy the following \emph{natural structure conditions}: for almost all $x\in\Omega$, we have
    \begin{equation}\label{eq:coupled-condi}
      \begin{cases}
        (\lvert w_1 \rvert+\lvert w_{1x} \rvert, \lvert w_2 \rvert+\lvert w_{2x} \rvert)\leq\Lambda(R)\left( V_1^{k_1}+V_2^{k_2-1},V_2^{k_2} \right);\\
        \left( \lvert q_1 \rvert + \lvert q_{1x} \rvert, \lvert q_2 \rvert + \lvert q_{2x} \rvert  \right)\leq\Lambda(R) \left( V_1^{k_1-1},V_2^{k_2-1} \right); \\
        \pi\cdot w_z\cdot \pi^T\leq\Lambda(R)\left( \sum\limits_{a=1}^2 V_a^{k_a}\lvert \pi_a \rvert^2+V_2^{k_2-1}\lvert \pi_1 \rvert^2 \right);\\
        \lvert q_z \rvert +\lvert w_p^T \rvert\leq \Lambda(R)
        \begin{pmatrix}
          V_1^{k_1-1}&0\\
          V_2^{k_2-2}&V_2^{k_2-1}
        \end{pmatrix};\\
        \lvert q_p \rvert\leq\Lambda(R)\mathrm{diag}\left( V_1^{k_1-2},V_2^{k_2-2} \right);\\
        \bar \pi\cdot q_p\cdot \bar \pi^T\geq\lambda(R)\sum\limits_{a=1}^2V_a^{k_a-2}\lvert \bar \pi_a \rvert^2;
      \end{cases}
    \end{equation}
    where the mixed derivatives with respect to $x$, $z$ and $p$ are simply denoted by subscripts; $\Lambda \mathpunct{:}=\Lambda(R)>\lambda \mathpunct{:}=\lambda(R)>0$ are constants depending on $R$, and $R>0$ is the upper bound of $(x,z)$, i.e., $\lvert x \rvert^2+\lvert z \rvert^2\leq R^2$; $V_a \mathpunct{:}=(1+\lvert p_a \rvert^2)^{1/2}$, $a=1,2$; $\pi=(\pi_1,\pi_2)$, $\pi_a=(\pi_a^i)_{m_a\times 1}$; $\bar \pi=(\bar \pi_1,\bar \pi_2)$, $\bar \pi_a=(\bar \pi_{a\alpha}^i)_{m_a\times n}$ are any constant matrices; Here we basically follow the notations of \cite{Morrey2008Multiple}*{(1.10.8$''$)}. In addition, $\lvert \cdot \rvert$ is the maximum norm, and for two $2\times 2$ block non-symmetric real matrices $M_1$, $M_2$, $\lvert M_1 \rvert\leq M_2$ means $|M_{1;ab}|\leq  M_{2;ab} $, for all $a,b=1,2$; similar notations are adopted for $1\times 2$ block matrices. In particular, the first condition for $q_p$ implies that $q_p=\begin{pmatrix}\partial_{p_1}q_1&\partial_{p_2}q_1\\\partial_{p_1}q_2&\partial_{p_2}q_2\end{pmatrix}$ is a block diagonal matrix, the second condition for $q_p$ is the ellipticity.
    \begin{mthm}\label{thm:Morrey}
      Suppose that $\Omega$ is a bounded domain in $\mathbb{R}^n$, $z=(z_1,z_2)$, $z_a\in L_1^{k_a}\cap C^\mu(\Omega, \mathbb{R}^{m})$, for some $0<\mu<1$ and $k_a\geq 2$, $a=1,2$, is a weak solution of the 2-coupled system \eqref{eq:2coupled-system}, with the coefficients $q=q(x,z,p)\in L^{1}_{1,\loc}\bigcap C^0(\Omega\times \mathbb{R}^{m}\times \mathbb{R}^{mn})$ and $w=w(x,z,p)\in L_{1,\loc}^1\bigcap C^0(\Omega\times \mathbb{R}^{m}\times \mathbb{R}^{mn})$, $m=m_1+m_2$, and satisfying the natural structure conditions \eqref{eq:coupled-condi}. If $z_1\in L_1^{2k_2}(\Omega, \mathbb{R}^{m_1})$, then $z\in L_{2}^2(\Omega, \mathbb{R}^{m})$.
    \end{mthm}
    \begin{rmk}
      On one hand, taking $z_1=0$ or taking $z_2=0$ and $k_2=k_1/ 2$ in~\autoref{thm:Morrey} gives Morrey's theorem~\cite{Morrey2008Multiple}*{Thm.~1.11.1$'$}. On the other hand, applying~\cite{Morrey2008Multiple}*{Thm.~1.11.1$'$} to a coupled system for $z=(z_1, z_2)$ does not simply imply the results in ~\autoref{thm:Morrey}, because the conditions~\eqref{eq:coupled-condi} given in~\autoref{thm:Morrey} are different from those in~\cite{Morrey2008Multiple}*{Thm.~1.11.1$'$} for a coupled system of $z=(z_1, z_2)$. In fact, there are two main differences. The first one is that the coupling relation in conditions \eqref{eq:coupled-condi} is expressed in terms of $w_1$, $w_{1x}$, $w_z$, $q_z$ and $w_p$, which will produce cross terms as expressed by the terms of the last parentheses in~\eqref{eq:l22-estimate}, see~\autoref{sec:couple-regularity}. To control these extra terms, we need to make additional regularity assumption for $z_1$, which is natural for coupled systems. The second one is that, the conditions \eqref{eq:coupled-condi} are only required to be held almost everywhere in $\Omega$ and the assumption on the regularity of the coefficients $q$ and $w$ is also weakened. We can check the well-definiteness of the weak solution in \eqref{eq:coupled-weak} under these regularity assumptions. The latter is useful when dealing with some coupled systems with non-smooth coefficients.
    \end{rmk}
    \begin{rmk}
      The coupling relation expressed by $w_1$, $w_{1x}$, $w_z$, $q_z$ and $w_p$ in conditions \eqref{eq:coupled-condi} is sharp and delicate in some sense. From the coupled condition of $w_1$, $w_{1x}$, $w_z$, it seems that one can add some lower order perturbed terms such as $V_2^{k_2-1}$, however, the coupled condition of $q_z$ and $w_p$ shows that this principle is not true anymore. This is because if we change the upper corner $0$ to $V_2^{k_2-2}$ or any other nonzero lower order term of $V_2$, it will then produce some new coupled terms, which cannot be analytically controlled anymore. For the same reason, the transpose of $w_p$ is also crucial here.
    \end{rmk}
    To get the regularity up to the boundary for $\alpha$-YMH fields satisfying the boundary condition (\hyperlink{bdry:N}{$\mathcal{N}$}), we shall locally reflect both the section $\phi$ and the connection $A$ across the free boundary naturally and derive a new coupled system for the reflected fields, then we apply the regularity results in~\autoref{thm:Morrey} to this new coupled system to get the interior regularity of the reflected fields, which gives the regularity up to the boundary of the original one.
    \bigbreak
    Next we study the existence of YMH fields under our boundary condition (\hyperlink{bdry:N}{$\mathcal{N}$}) by exploring the limiting behavior of a sequence of $\alpha$-YMH fields as $\alpha\to1$. Since the Dirichlet energy $\Vert\nabla_A\phi \Vert_{L^2}^2$ is conformally invariant in dimension two, energy concentration and bubbling phenomena can possibly occur, which is similar to various harmonic map type problems.  Actually, in \citelist{\cite{Song2011CriticalYangMillsHiggs}\cite{Song2016Convergence}}, it was shown that when the surface $\Sigma$ is closed, a sub-sequence of the $\alpha$-YMH fields converges to a YMH fields away from at most finitely many blow-up points where the energies concentrate. At each blow-up point, a harmonic sphere can split off. In the situation considered in this paper,  where $\Sigma$ has non-empty boundary, it is sufficient to focus on the blow-up behavior near the boundary $\partial\Sigma$.  For $\alpha$-harmonic maps with free boundary, we refer to~\cite{Fraser2000free}.

    Our main result, in analogy to the closed case (see \cite{Song2011CriticalYangMillsHiggs}), is the following boundary bubbling convergence theorem for a sequence of $\alpha$-YMH fields under our boundary conditions.
    \begin{mthm}\label{thm:blowup}
      There exists a constant $\alpha_0>1$, such that if $\left\{ (A_{\alpha},\phi_{\alpha}) \right\} \subset \mathscr{A} \times \mathscr{S}_K$ is a sequence of  smooth $\alpha$-YMH fields with $\alpha\in(1,\alpha_0)$ and $\mathcal{L}_\alpha(A_{\alpha},\phi_{\alpha})\leq \Lambda<+\infty$, then the blow-up set $\mathcal{S}$ of $\left\{ (A_{\alpha},\phi_{\alpha}) \right\}$ defined by
      \[
        \mathcal{S}\mathpunct{:}=\left\{ x\in\Sigma:\lim_{r\to0}\liminf_{\alpha\to1}\int_{U_r(x)}\lvert \nabla_{A_{\alpha}}\phi_{\alpha} \rvert^2\geq\eps_0 \right\},
      \]
      is a set of at most finitely points; where $\epsilon_0>0$ is a constant depending on the geometry of the bundle (see~\autoref{lem:eps-regulairty}) and $U_r(x)$ is a geodesic ball of radius $r$ centered at $x$ in $\Sigma$. Moreover, as $\alpha\to1$, after taking a sub-sequence of $\left\{ (A_\alpha,\phi_\alpha) \right\}$, we have
      \begin{enumerate}
        \item\label{item:prop:blowup} $A_{\alpha}\to A_\infty$  in $C^{\infty}_{\mathrm{loc}}(\Sigma\setminus \mathcal{S})\cap C^0(\Sigma)$ and $\phi_{\alpha} \to \phi_\infty$  in $C_{\mathrm{loc}}^\infty(\Sigma\setminus \mathcal{S})$ module gauge. Moreover, $(A_\infty,\phi_\infty)$ extends to a smooth YMH fields on $\Sigma$ satisfying the boundary condition (\hyperlink{bdry:N}{$\mathcal{N}$}).
      \item For each $x\in \mathcal{S}\cap \partial \Sigma$, there exist either a non-trivial harmonic spheres $\omega \mathpunct{:} S^2\to F$ or a non-trivial harmonic discs $w \mathpunct{:} B\to F$ with free boundary on $K$.
          %such that
          %\[
          %  \lim_{r\to0}\limsup_{\alpha\to1}\int_{U_r(x)}\lvert \nabla_{A_\alpha}\phi_\alpha \rvert^2\geq\sum_{i=1}^k\int_{S^2}\lvert d\omega_i \rvert^2+\sum_{j=1}^l\int_{B}\lvert dw_j \rvert^2.
          %\]
      \end{enumerate}
    \end{mthm}
    \begin{rmk}
      In \autoref{thm:blowup}, if $F$ admits no non-trivial harmonic 2-spheres, then either $(A_\alpha,\phi_\alpha)$ subconverges smoothly to a YMH field $(A_\infty,\phi_\infty)$ over $\Sigma$, where $\phi_\infty$ and $\phi_\alpha$ are in the same homotopy class, or there exists at least one minimal 2-disc in $F$ with free boundary on $K$.
    \end{rmk}

    The rest of the paper is organized as follows. In~\autoref{sec:PS-alpha}, we study the perturbed YMH functional and $\alpha$-YMH fields. We start with the verification of Palais--Smale condition in~\autoref{sec:PS}, then prove the regularity~\autoref{thm:Morrey} in~\autoref{sec:couple-regularity}, from which~\autoref{thm:alpha-smoothness} follows in~\autoref{sec:smoothness}. In~\autoref{sec:eps-regularity}, we derive local estimates for both the connection and the section. The blow-up argument is demonstrated in~\autoref{sec:blowup}, which is the content of~\autoref{thm:blowup}. Finally, we collect some classical boundary estimates and regularity theorems of free boundary problems in~\autoref{sec:bdry-regularity}.

    \section{The \texorpdfstring{$\alpha$}{alpha}-YMH functional}\label{sec:PS-alpha}
    We first show in~\autoref{sec:PS} that $\mathcal{L}_\alpha$, $\alpha>1$, satisfies the Palais--Smale condition so that there exist critical points of $\mathcal{L}_\alpha$ which solve the Euler--Lagrange equation of $\mathcal{L}_\alpha$ weakly. To improve the regularity of the weak solution, we generalize a classical regularity result of elliptic systems to coupled elliptic systems in~\autoref{sec:couple-regularity} and then rewrite the weak solution into strong form, from which the smoothness of the solution when $\alpha>1$ follows from classical elliptic estimate (up to the boundary) and bootstrap as sketched in~\autoref{sec:smoothness}.
    \subsection{The Palais--Smale condition}\label{sec:PS}
    It is well-known that the Palais--Smale condition is crucial in deriving the existence of certain kinds of critical points in variational problems. For $\alpha$-harmonic maps, we refer to~\cite{Urakawa1993Calculus}*{Sect.~3.2} for the case of closed surfaces and~\cite{Fraser2000free}*{Prop.~1.1} for the free boundary case. The same idea is applied to $\alpha$-YMH functional in~\cite{Song2011CriticalYangMillsHiggs}*{Lem.~3.2} for the case of a closed surface $\Sigma$. In what follows, we verify the Palais--Smale condition for $\mathcal{L}_\alpha$ when $\partial\Sigma\neq\emptyset$ and the boundary condition (\hyperlink{bdry:N}{$\mathcal{N}$}) is imposed.

    Recall the following weak compactness theorem of connections on manifolds with boundary.
    \begin{thm}[\cite{Wehrheim2004Uhlenbeck}*{Thm.~7.1, p.~108}]\label{thm:weak-compactness}
      Suppose $M$ is a Riemannian manifold with boundary. Let $2p>\dim M\geq2$ and $\left\{ A_n \right\}\subset\mathscr{A}_1^p$ be a sequence of connections with $\lVert F_{A_n}\rVert_{L^p(M)}^p\leq \Lambda<+\infty$. Then, there exists a sub-sequence, still denoted by $\left\{ A_n \right\}$, and a sequence of gauge transformations $S_n\in \mathscr{G}_2^p$ such that $\left\{ S^*_nA_n \right\}$ converges weakly in $\mathscr{A}_1^p$. That is, a sub-sequence of $\left\{ A_n \right\}$ converges weakly in $\mathscr{A}_1^p$ module gauge.
    \end{thm}
    With the help of above theorem, we will show that $\mathcal{L}_\alpha$, $\alpha>1$, satisfies the Palais--Smale condition.
    \begin{lem}\label{lem:ps}
      For any $\alpha>1$, $\mathcal{L}_\alpha$ satisfies the Palais--Smale condition on the product space $\mathscr{A}_1^2\times \mathscr{S}_{1,K}^{2\alpha}$. That is, for any sequence $\left\{ (A_n,\phi_n) \right\}\in \mathscr{A}_1^2\times \mathscr{S}_{1,K}^{2\alpha}$, if
      \begin{enumerate}
        \item\label{item:1-lem-ps} $\mathcal{L}_\alpha(A_n,\phi_n)\leq \Lambda<+\infty$;
        \item\label{item:2-lem-ps} $\lVert D \mathcal{L}_\alpha(A_n,\phi_n)\rVert\to0$, where the norm is taken in $T^*_{(A_n,\phi_n)} \mathscr{A}_1^2\times \mathscr{S}_{1,K}^{2\alpha}$;
      \end{enumerate}
      then there exists a sub-sequence which converges strongly in $\mathscr{A}_1^2\times \mathscr{S}_{1,K}^{2\alpha}$ module gauge.
    \end{lem}
    \begin{proof}
      In what follows, for simplicity, we don't distinguish a sequence and its sub-sequences.
      \step We first show that $\left\{ A_n \right\}$ converges strongly in $\mathscr{A}_1^2$ to some $A_\infty$.

      By assumption~\eqref{item:1-lem-ps} of $\mathcal{L}_\alpha$, $\lVert F_{A_n}\rVert_{L^2(\Sigma)}^2\leq \Lambda$ and we can apply~\autoref{thm:weak-compactness} to show that $\left\{ S_n^*A_n \right\}$ converges weakly to $A_\infty$ in $\mathscr{A}_1^2$ for some sequence $\left\{ S_n \right\}\subset\mathscr{G}_2^2$. For simplicity, we still denote $\left\{ S^*_nA_n \right\}$ by $\left\{ A_n \right\}$, then
      \begin{equation}\label{eq:An-A-infty-bdd}
        A_n\weakto A_\infty\xtext{in} \mathscr{A}_1^2,\quad\lVert A_n-A_0 \rVert_{L_1^2}\leq C,
      \end{equation}
      where $A_0$ is the reference connection. By assumption~\eqref{item:2-lem-ps} of $\mathcal{L}_\alpha$,
      \begin{equation}\label{eq:DL-alpha}
        |\left\langle D\mathcal{L}_\alpha(A_n,\phi_n),(A_n-A_\infty,0) \right\rangle|\leq\lVert D\mathcal{L}_\alpha(A_n,\phi_n)\rVert\cdot \lVert A_n-A_\infty\rVert_{L_1^2} \to0.
      \end{equation}
      Similar to the computation of Euler--Lagrange equation of $\mathcal{L}_\alpha$, we have
      \begin{align*}
        \left\langle D \mathcal{L}_\alpha(A_n,\phi_n),(A_n-A_\infty,0) \right\rangle
    &=\int_{\Sigma}\left\langle F_{A_n},D_{A_n}(A_n-A_\infty)\right\rangle\\
    &\qquad+\int_\Sigma\left\langle\alpha(1+|\nabla_{A_n}\phi_n|^2)^{\alpha-1}\nabla_{A_n}\phi_n,(A_n-A_\infty)\phi_n \right\rangle\\
    &=\mathpunct{:}\I+\II.
      \end{align*}
      Since $A_n-A_\infty\weakto 0$ in $L_1^2$, it follows that $\lVert A_n-A_\infty\rVert_{L^q}\to0$ for any $1\leq q<+\infty$ by the Sobolev embedding theorems. Now, by H\"older's inequality
      \begin{align*}
        \lvert \II \rvert&\leq\alpha\lVert 1+|\nabla_{A_n}\phi_n|^2\rVert_{L^{\alpha}}^{\alpha-1}\cdot\lVert \nabla_{A_n}\phi_n\rVert_{L^{2\alpha}}
        \cdot\lVert \phi_n\rVert_{L^\infty}\cdot\lVert A_n-A_\infty\rVert_{L^{2\alpha}}\\
                         &\leq C(\Lambda)\lVert A_n-A_\infty\rVert_{L^{2\alpha}}\to0,
      \end{align*}
      as $n\to\infty$. For $\I$, we can compute, for $a_n=A_n-A_0\in \Omega^1(\mathfrak{g}_{\mathcal{P}})$ and $a_\infty=A_\infty-A_0\in \Omega^1(\mathfrak{g}_{\mathcal{P}})$,
      \begin{align*}
        \I&=\int_\Sigma\left\langle F_{A_0}+D_{A_0}a_n+a_n\wedge a_n,D_{A_0}(a_n-a_\infty)+[a_n\wedge(a_n-a_\infty)] \right\rangle\\
          &=\int_\Sigma|D_{A_0}(a_n-a_\infty)|^2+\int_\Sigma\left\langle D_{A_0}a_\infty+a_n\wedge a_n,D_{A_0}(a_n-a_\infty) \right\rangle\\
          &\qquad+\int_\Sigma\left\langle F_{A_n},[a_n\wedge(a_n-a_\infty)] \right\rangle.
      \end{align*}
      Note that $D_{A_0} a_\infty\in L^2$, $\lVert F_{A_n}\rVert_{L^2}^2<\Lambda$, and~\eqref{eq:An-A-infty-bdd} implies that $\lVert a_n\wedge a_n\rVert_{L^2}<C\lVert a_n\rVert_{L^4}^2<C'\lVert A_n-A_0 \rVert_{L_1^2}^2<C''$ by the Sobolev embedding. Thus, by the definition of weak convergence and the H\"older's inequality, we know that the last two terms in $\I$ tend to $0$ as $n\to\infty$ and
      \[
        \I\to\lVert D_{A_0}(A_n-A_\infty)\rVert_{L^2}.
      \]
      Inserting the estimates of $\I$ and $\II$ into~\eqref{eq:DL-alpha}, we obtain that
      \[
        \lVert D_{A_0}(A_n-A_\infty)\rVert_{L^2}\to0.
      \]
      Since $A_n-A_\infty\weakto 0$ in $L_1^2$ and $A_n-A_\infty\to0$ strongly in $L^2$, we conclude that $A_n\to A_\infty$ strongly in $\mathscr{A}_1^2$.
      \step Next, we show that for fixed $A_\infty\in \mathscr{A}_1^2$, $\mathcal{L}_\alpha(A_\infty,\cdot)$ satisfies the Palais--Smale condition in $\mathscr{S}_{1,K}^{2\alpha}$. Recall, for $\alpha>1$, $\mathscr{S}_{1,K}^{2\alpha}$ is defined as a subspace of $L_1^{2\alpha}(\mathcal{E})$,
      \[
        \mathscr{S}_{1,K}^{2\alpha}\mathpunct{:}=\left\{ \phi\in L_1^{2\alpha}(\mathcal{E}):\phi(x)\in \mathcal{F}\xtext{for a.e. $x\in \Sigma$}\xtext{and}\phi(x)\in \mathcal{K}\xtext{for a.e. $x\in\partial\Sigma$} \right\},
      \]
      We note the following facts:
      \begin{itemize}
        \item As a closed sub-manifold of Banach manifold $L_1^{2\alpha}(\mathcal{E})$, $\mathscr{S}_{1,K}^{2\alpha}$ can be given a structure of smooth Banach manifold. In particular, $\mathscr{S}_{1,K}^{2\alpha}$ is complete under the pull-back Finsler metric $\lVert \cdot \rVert_{L_1^{2\alpha}(\mathcal{E})}$.
        \item The tangent space of $\mathscr{S}_{1,K}^{2\alpha}$ is given by
          \[
            T_\phi \mathscr{S}_{1,K}^{2\alpha}=\left\{ \psi\in L_1^{2\alpha}(\phi^*T \mathcal{F}^v):\psi(x)\in T_{\phi(x)}\mathcal{K}^v\xtext{for a.e. $x\in\partial\Sigma$} \right\}.
          \]
      \end{itemize}
      Now define
      \[
        \mathscr{L}\mathpunct{:}L_1^{2\alpha}(\mathcal{E})\to \mathbb{R},\quad
        \mathscr{L}(\phi)=\int_{\Sigma}(1+|\nabla_{A_\infty}\phi|^2)^{\alpha}+|F_{A_\infty}|^2+|\mu(\phi)|^2,
      \]
      where we extend $\mu$ to the sections of $\mathcal{E}$ by $\tilde{\mu}(\phi)\mathpunct{:}=\eta(x)\cdot \mu\bigl(\pi_N(\phi)\bigr)$, here $\pi_N$ is the nearest projection from a neighborhood $\mathcal{N}$ of $\mathcal{F}$ in $\mathcal{E}$ to $\mathcal{F}$, and $\eta$ is a cutoff function supported on $\mathcal{N}$ and equals to $1$ when restricted to $\mathcal{F}$. Clearly, $\mathscr{L}(\phi)=\mathcal{L}_\alpha(A_\infty,\phi)$ when we restrict $\phi$ to $\mathscr{S}_{1,K}^{2\alpha}$, which is denoted by $\mathcal{J}(\phi)$. We can imitate the argument of~\cite{Urakawa1993Calculus}*{Sect.~3.2, p.~105ff} to show
      \begin{enumerate}
        \item $\mathcal{J}$ is a $C^2$ function on $\mathscr{S}_{1,K}^{2\alpha}$.
        \item There exists a positive constant $C$ (depending on $\mu$) such that for any $\phi_1,\phi_2\in L_1^{2\alpha}(\mathcal{E})$,
          \begin{equation}\label{eq:claim-eng-control}
            (D \mathscr{L}_{\phi_1}-D \mathscr{L}_{\phi_2})(\phi_1-\phi_2)\geq
            \begin{multlined}[t]
              C\bigl( \lVert \phi_1-\phi_2\rVert_{L_1^{2\alpha}}^{2\alpha}-\lVert \phi_1-\phi_2\rVert_{L^{2\alpha}}^{2\alpha}-\lVert \phi_1-\phi_2\rVert_{L^2}^2 \bigr).
            \end{multlined}
          \end{equation}
        \item Suppose $\left\{ \phi_n \right\}\subset \mathscr{S}_{1,K}^{2\alpha}$ is a bounded sequence under the norm of $L_1^{2\alpha}(\mathcal{E})$, then there exists a sub-sequence such that
          \[
            \lVert (\id-\Pi_{\phi_n})(\phi_n-\phi_m)\rVert_{L_1^{2\alpha}} \to0,\quad \xtext{as} m,n\to0,
          \]
          where $\id$ is the identity map, and $\Pi_{\phi_n}$ is the fiber-wise orthogonal projection from $\mathcal{E}$ to $T\mathscr{S}_{1,K}^{2\alpha}$ at $\phi_n$. Here we should be careful about the projection at the boundary. As we require that the projected section lies in $T_{\phi_n} \mathscr{S}_{1,K}^{2\alpha}$, which requires, at the boundary, it is a vector of $T_{\phi_n(x)} \mathcal{K}^v$. This is accomplished by first defining the projections
          \begin{align*}
            \Pi^E_\phi \mathpunct{:} L_1^{2\alpha}(\mathcal{E})&\to L_1^{2\alpha}(\phi^*T \mathcal{F}^v)\\
            \psi&\mapsto\Pi^E_\phi(\psi),\quad [\Pi^E_\phi(\psi)](x)\mathpunct{:}=\Pi^E(\phi(x))(\psi(x)),
          \end{align*}
          and
          \begin{align*}
            \Pi_\phi^{K^\perp} \mathpunct{:}L_1^{2\alpha}(\mathcal{E})&\to L_1^{2\alpha}(\phi^*T \mathcal{K}^v)\\
            \psi&\mapsto\Pi_\phi^{K^\perp}(\psi),\quad[\Pi_\phi^{K^\perp}(\psi)](x)\mathpunct{:}=\Pi^{K^\perp}(\phi(x))(\psi(x)).
          \end{align*}
          Here $\Pi^E(y)$ denotes the orthogonal projection of the fiber $\mathcal{E}_y$ onto the tangent space $T_y \mathcal{F}^v$ for $y\in \mathcal{F}^v$, and $\Pi^{K^\perp}(y)$ denotes the orthogonal projection of $\mathcal{E}_y$ onto $T_y^\perp \mathcal{K}^v$ for $y\in \mathcal{K}^v$ with respect to the decomposition $\mathcal{E}_y=T_y \mathcal{K}^v\oplus T_y^\perp \mathcal{K}^v$. Then our real projection $\Pi_\phi$ is defined by
          \begin{align*}
            \Pi_\phi \mathpunct{:} L_1^{2\alpha}(\mathcal{E})&\to T_\phi \mathscr{S}_{1,K}^{2\alpha}\\
            \psi&\mapsto\Pi_\phi(\psi)\mathpunct{:}=\Pi_\phi^E\left( \psi- \theta\left( \Pi_\phi^{K^\perp}(\psi|_{\partial\Sigma}) \right)\right),
          \end{align*}
          where $\theta \mathpunct{:} L_{1-1/(2\alpha)}^{2\alpha}(\mathcal{E}|_{\partial\Sigma})\to L_1^{2\alpha}(\mathcal{E})$ is a continuous linear extension operator from the trace space of $L_1^{2\alpha}(\mathcal{E})$.
        \item There exists a sub-sequence of $\left\{ \phi_n \right\}$, such that
          \begin{equation}\label{eq:claim-eng}
            D \mathscr{L}_{\phi_n}(\phi_n-\phi_m)\to0,\quad\xtext{as}m,n\to\infty.
          \end{equation}
      \end{enumerate}
      Now, we continue the verification of Palais--Smale condition of $\mathcal{L}_\alpha(A_\infty,\cdot)$. Since $\mathcal{L}_\alpha(A_n,\phi_n)\leq\Lambda$ and $A_n\to A_\infty$ strongly in $\mathscr{A}_1^2$, we know that $\mathcal{L}_\alpha(A_n,\phi_n)\to \mathcal{L}_\alpha(A_\infty,\phi_n)$ as $n\to\infty$. In particular, $\left\{ \phi_n \right\}$ is bounded in $L_1^{2\alpha}(\mathcal{E})$. By~\eqref{eq:claim-eng}, we can choose a sub-sequence such that $D \mathscr{L}_{\phi_n}(\phi_n-\phi_m)\to0$ as $m,n\to\infty$. It is clear that
      \[
        (D \mathscr{L}_{\phi_n}-D \mathscr{L}_{\phi_m})(\phi_n-\phi_m)\to0,\quad\xtext{as} m,n\to\infty.
      \]
      Note also that $\left\{ \phi_n \right\}$ is bounded in $L_1^{2\alpha}(\mathcal{E})$, the weak compactness of this Sobolev space implies that there is a convergent sub-sequence and so, for such a sequence we have
      \[
        \lVert \phi_n-\phi_m\rVert_{L^{2\alpha}}\to0\quad\xtext{and}\quad\lVert \phi_n-\phi_m\rVert_{L^2}\to0,\quad\xtext{as}m,n\to\infty.
      \]
      Thus,~\eqref{eq:claim-eng-control} implies that
      \[
        \lVert \phi_n-\phi_m\rVert_{L_1^{2\alpha}}\to0,\quad\xtext{as}m,n\to\infty,
      \]
      i.e., $\left\{ \phi_n \right\}$ is a Cauchy sequence in $\mathscr{S}_{1,K}^{2\alpha}$. It is convergent because $\mathscr{S}_{1,K}^{2\alpha}$ is complete. As we have shown that $A_n$ converges to $A_\infty$ in $\mathscr{A}_1^2$ up to sub-sequence, and for such sub-sequence, $\left\{ \phi_n \right\}$ converges to some $\phi_\infty$ in $\mathscr{S}_{1,K}^{2\alpha}$, it follows that $\mathcal{L}_\alpha$ satisfies the Palais--Smale condition.
    \end{proof}
    \subsection{A regularity theorem for coupled equations}\label{sec:couple-regularity}
    We prove in this section a regularity theorem for H\"older continuous weak solution of some coupled equations, which is an extension of the classical regularity results by Ladyzhenskaya-Ural\cprime ceva and Morrey. The idea is that, when the coupling relations of the coupled system satisfies the conditions in~\eqref{eq:coupled-condi}, then the bad terms appeared due to the coupling relations are controllable.
    \begin{proof}[Proof of~\autoref{thm:Morrey}]
      For any fixed $x_0\in\Omega$, by the relation between weak derivatives and difference quotients, we only need to show the uniform boundedness of $\left\{ \lVert \nabla z_h \rVert_{L^2(B_r)} \right\}_{0<h<r}$, where $B_r=B_r(x_0)$, and $r>0$ is a small real number to be determined latter. Currently, we only assume $0<8r<\mathrm{dist}(x_0,\partial \Omega)=r_0$. $z_h$ is the difference quotient defined as follows: for any fixed coordinate direction $e_\gamma$ and real number $h$, $0<\lvert h \rvert<r$,
      \[
        z_h=(z_{1h},z_{2h}), \quad z_{ah}^i \mathpunct{:}=\Delta_\gamma^h z^i \mathpunct{:}=\frac{z_a^i(x+he_\gamma)-z_a^i(x)}{h},\quad a=1,2,\, i=1,2,\ldots,m_a.
      \]
      Now, let us fix some $D'\subset\!\subset \Omega_r \mathpunct{:}=\left\{ x\in\Omega| \mathrm{dist}(x,\partial\Omega)\geq r \right\}$, $x_0\in D'$ and let $\xi=(\xi_1,\xi_2)$ be a test function with $\mathrm{supp}\,\xi\subset\!\subset D'$. We denote by $\xi_h$ the difference quotient of $\xi$ and substitute $\xi$ in~\eqref{eq:coupled-weak} by $\xi_{-h}$, hereafter the repeated indices $\alpha$, $\beta$, $a$, $b$, $i$ and $j$ are summed,
      \begin{equation}\label{eq:coupled-weak-diff}
        0=\int_{D'}\partial_\alpha\xi_a^i\Delta_\gamma^hq_{ai}^\alpha+\xi_{a}^i\Delta_\gamma^h w_{ai}.
      \end{equation}
      If we set $\Delta x \mathpunct{:}=he_\gamma$, $\Delta z=z(x+\Delta x)-z(x)$ and $\Delta p=p(z(x+\Delta x))-p(z(x))$, then since $q_a\in L_{1,\loc}^1\cap C^0(\Omega\times \mathbb{R}^m\times \mathbb{R}^{mn})$ and $z_a\in L_1^{k_a}(\Omega)$, by the fundamental theorem of calculus for distributions (see \cite{LiebLoss2001Analysis}*{Thm.~6.9}) and the chain rule (see \cite{LiebLoss2001Analysis}*{Thm.~6.16}), for a.e. $x\in\Omega$, we have
      \begin{align*}
        \Delta_\gamma^hq_{ai}^\alpha&= \frac{1}{h}\int_0^1 \frac{d}{dt}q_{ai}^\alpha\left(x+t\Delta x,z(x)+t\Delta z,p(z(x))+t\Delta p\right)dt\\
                                    &=\int_0^1\left( \partial_\gamma q_{ai}^\alpha(t)+\partial_{z_b^j}q_{ai}^\alpha(t)z_{bh}^j+\partial_{p_{b\beta}^j}q_{ai}^\alpha(t)\partial_\beta z_{bh}^j \right)dt,\quad a=1,2,
      \end{align*}
      where $q_{ai}^\alpha(t) \mathpunct{:}=q_{ai}^\alpha(x+t\Delta x,z(x)+t\Delta z, p(z(x))+t\Delta p)$. Define $w_{ai}(t)$ similarly, we have
      \[
        \Delta_\gamma^h w_{ai}=\int_0^1\left( \partial_\gamma w_{ai}(t)+\partial_{z_b^j}w_{ai}(t)z_{bh}^j+\partial_{p_{b\beta}^j}w_{ai}(t)\partial_\beta z_{bh}^j \right)dt,\quad a=1,2.
      \]
      Therefore, we can rewrite~\eqref{eq:coupled-weak-diff} in matrix form as
      \begin{equation}\label{eq:coupled-weak-diff'}
        0=\int_{D'}\int_0^1\nabla\xi\cdot\bigl( q_p(t)\cdot\nabla z_h + q_z(t)\cdot z_h+ q_x(t) \bigr)+\int_{D'}\int_0^1\xi\cdot\bigl( w_p(t)\cdot\nabla z_h+w_z(t)\cdot z_h+w_x(t) \bigr).
      \end{equation}
      Next, let $\eta\in C_0^\infty(\Omega)$ be a cutoff function satisfying
      \[
        0\leq\eta\leq1,\quad\eta|_{D}\equiv1,\quad \mathrm{supp}\,\eta\subset\!\subset D_r',\quad\lvert \nabla\eta \rvert\leq 8/r,
      \]
      where $D\subset\!\subset D_{r}'\subset\!\subset D'\subset\!\subset\Omega_r$, and $D_{r}' \mathpunct{:}=\left\{ x\in D'|\mathrm{dist}(x,\partial D')\geq r \right\}$. If we set $Z_a^i=\eta z_{ah}^i$ and $\xi_a^i=\eta Z_a^i$, then
      %\[
      %  \partial_\alpha\xi_a^i=\eta\left( \partial_\alpha Z_a^i +\partial_\alpha\eta z_{ah}^i\right),\quad
      %  \eta \partial_\alpha z_{ah}^i=\partial_\alpha Z_a^i-\partial_\alpha\eta z_{ah}^i.
      %\]
      %That is,
      \[
        \eta\nabla z_h=\nabla Z-\nabla\eta z_h, \quad \nabla\xi=\eta(\nabla Z+\nabla\eta z_h).
      \]
      By~\eqref{eq:coupled-weak-diff'}, and note that we did not assume $q_p$ is symmetric (i.e., $\partial_{p_{b\beta}^j}q_{ai}^\alpha\neq\partial_{p_{a\alpha}^i}q_{bj}^\beta$ in general),
      %  \begin{multline*}
      %    0=\int_{D'}\biggl\{ \left( \nabla Z+\nabla\eta z_h \right)\int_0^1\left( q_p(t)\left( \nabla Z^T-\nabla\eta z_h^T \right)+q_z(t) Z^T+q_x(t) \right)\\
      %      +Z\int_0^1\left( w_p(t)\left( \nabla Z^T-\nabla\eta z_h^T \right)+ w_z(t) Z^T+w_x(t) \right)
      %    \biggr\},
      %  \end{multline*}
      %  i.e.,
      \begin{equation}\label{eq:dz-qp-dz}
        \begin{split}
      &\int_{D'}\left[ \nabla Z\cdot\int_0^1q_p(t)dt \cdot\nabla Z \right]\\
      &\qquad=\int_{D'}\nabla Z\cdot\int_0^1 q_p(t)dt\cdot\nabla\eta z_h-\int_{D'}\nabla\eta z_h\cdot\int_0^1 q_p(t)\cdot \nabla Z\\
      &\qquad\qquad+\int_{D'}\left[ \nabla\eta z_h\cdot\int_0^1q_p(t)dt\cdot\nabla\eta z_h \right]
      -\int_{D'}\left[ (\nabla Z+\nabla\eta z_h)\cdot\int_0^1q_z(t)dt\cdot Z \right]\\
      &\qquad\qquad-\int_{D'}\left[ \eta(\nabla Z+\nabla\eta z_h)\cdot\int_0^1q_x(t)dt \right]
      -\int_{D'}\left[ Z\cdot\int_0^1w_p(t)dt\cdot(\nabla Z-\nabla\eta z_h) \right]\\
      &\qquad\qquad-\int_{D'}\left[ Z\cdot\int_0^1w_z(t)dt\cdot Z \right]
      -\int_{D'}\left[ \eta Z\cdot\int_0^1w_x(t)dt \right].
        \end{split}
      \end{equation}
      To simplify the notations, let us set
      \[
        A_{ah} \mathpunct{:}=\int_0^1V_a^{k_a-2}(t)dt,\quad A_{ah}P_{ah} \mathpunct{:}=\int_0^1V_a^{k_a-1}(t)dt,\quad A_{ah}Q_{ah} \mathpunct{:}=\int_0^1V_a^{k_a}(t)dt,
      \]
      where $V_a(t)\mathpunct{:}=(1+\lvert p_a(z(x))+t\Delta p_a \rvert^2)^{1/2}$. Clearly,
      \[
        A_{ah}\geq1,\quad P_{ah}\geq1,\quad Q_{ah}\geq1,\quad P_{ah}^2\leq Q_{ah}.
      \]
      Although our condition \eqref{eq:coupled-condi} is only satisfied almost everywhere on $\Omega$, we essentially use these conditions in integral form and the value on a subset of measure zero will not affect the result. The ellipticity condition in~\eqref{eq:coupled-condi} implies
      \[
        \int_{D'}\left[ \nabla Z\cdot\int_0^1q_p(t)dt\cdot\nabla Z \right]\geq\lambda(R)\int_{D'}A_{ah}\lvert \nabla Z_a \rvert^2,
      \]
      where $R$ is the upper bound for $x\in D'$ and $z=z(x)$, i.e., $\lvert x \rvert^2+\lvert z \rvert^2\leq R^2$.
      The right-hand side terms of \eqref{eq:dz-qp-dz} can be controlled by condition~\eqref{eq:coupled-condi} and Cauchy-Schwarz inequality. Here, we demonstrate the estimates of the terms $q_z$ and $w_p$, which explains that the coupling structure of $q_z$ and $w_p$ in \eqref{eq:coupled-condi} is crucial.
      \begin{align*}
        -\int_{D'}\left[ \nabla Z\cdot\int_0^1q_z(t)dt\cdot Z \right]
     &\leq\Lambda(R)\int_{D'}\left[ A_{ah}P_{ah}\lvert \nabla Z_a \rvert\lvert Z_a \rvert+A_{2h} \lvert \nabla Z_2 \rvert\lvert Z_1 \rvert \right] \\
     &\leq\Lambda(R)\int_{D'}\left[ \epsilon A_{ah}\lvert \nabla Z_a \rvert^2+\frac{1}{4\epsilon}A_{ah}P_{ah}^2\lvert Z_a \rvert^2+ \frac{1}{4\epsilon}A_{2h}\lvert Z_1 \rvert^2 \right]\\
     -\int_{D'}\left[ Z\cdot\int_0^1w_p(t)dt\cdot\nabla Z \right]
     &\leq\Lambda(R)\int_{D'}\left[ \epsilon A_{ah}\lvert \nabla Z_a \rvert^2+\frac{1}{4\epsilon}A_{ah}P_{ah}^2\lvert Z_a \rvert^2+\frac{1}{4\epsilon}A_{2h}\lvert Z_1 \rvert^2 \right].
     %-\int_{D'}\left[ \nabla\eta z_h\cdot\int_0^1q_z(t)dt\cdot Z \right]
     %&\leq\Lambda(R)\int_{D'}A_{ah}P_{ah}\lvert \nabla\eta z_{ah} \rvert\lvert Z_a \rvert+A_{2h}\lvert \nabla\eta z_{2h} \rvert\lvert Z_1 \rvert \\
     %&\leq\frac{\Lambda(R)}{2}\int_{D'}A_{ah}\lvert \nabla\eta\rvert^2\lvert z_{ah} \rvert^2+A_{ah}P_{ah}^2\lvert Z_a \rvert^2+A_{2h}\lvert Z_1 \rvert^2\\
     %  \int_{D'}\left[ Z\cdot\int_0^1w_p(t)dt\cdot \nabla\eta z_h \right]
     %  &\leq\frac{\Lambda(R)}{2}\int_{D'}A_{ah}\lvert \nabla\eta\rvert^2\lvert z_{ah} \rvert^2+A_{ah}P_{ah}^2\lvert Z_a \rvert^2 + A_{2h}\lvert Z_1 \rvert^2\\
     %  -\int_{D'}\left[ \eta(\nabla Z+\nabla\eta z_h)\cdot\int_0^1e_\gamma\cdot q_x(t)dt \right]
     %&\leq\Lambda(R)\int_{D'}A_{ah}P_{ah}\lvert \nabla Z_a \rvert+A_{ah}P_{ah}\lvert \nabla\eta z_{ah} \rvert\\
     %&\leq\Lambda(R)\int_{D'}\epsilon A_{ah}\lvert \nabla Z_a \rvert^2+\frac{1}{4\epsilon}A_{ah}P_{ah}^2\\
     %  &\qquad+\frac{\Lambda(R)}{2}\int_{D'}A_{ah}\lvert \nabla\eta \rvert^2\lvert z_{ah} \rvert^2+A_{ah}P_{ah}^2\\
     %  \int_{D'}\left[ Z\cdot\int_0^1w_z(t)dt\cdot Z \right]
     %  &\leq\Lambda(R)\int_{D'}\left( A_{1h}Q_{1h}+A_{2h}P_{2h} \right)\lvert Z_1 \rvert^2+A_{2h}Q_{2h}\lvert Z_2 \rvert^2\\
     %  &\leq\Lambda(R)\int_{D'}A_{2h}P_{2h}\lvert Z_1 \rvert^2+A_{ah}Q_{ah}\lvert Z_a \rvert^2\\
     %  \int_{D'}\left[ \eta Z\cdot\int_0^1e_\gamma\cdot w_x(t)dt \right]
     %  &\leq\Lambda(R)\int_{D'}A_{2h}P_{2h}\lvert Z_1 \rvert+A_{ah}Q_{ah}\lvert Z_a \rvert.
   \end{align*}
   Therefore (recall that $\Lambda=\Lambda(R)$ and $\lambda=\lambda(R)$),
   \begin{equation}\label{eq:AQZ2}
     \int_{D'}A_{ah}\lvert \nabla Z_a \rvert^2
     \leq C(\Lambda,\lambda)\int_{D'}\left[ A_{ah}\lvert \nabla\eta \rvert^2\lvert z_{ah} \rvert^2
     +A_{ah}Q_{ah}\left( 1+\lvert Z_a \rvert^2 \right)+A_{2h}P_{2h} \lvert Z_1 \rvert^2 \right].
   \end{equation}

   Now, we need the following claim to handle $\int_{D'}A_{ah}Q_{ah}\lvert Z_a \rvert^2$.
   \begin{claim}[\citelist{\cite{Morrey2008Multiple}*{Lem.~5.9.1}\cite{LadyzhenskaiaUraltzeva1961smoothness}*{Lem.~2}}]
     With the assumption of \autoref{thm:Morrey}, we have for any $\delta>0$ and any $x_0\in\Omega$, there exists $\rho$, $0<\rho\leq 4r$, where $0<8r<r_0$, depending on $\delta$, $\lvert x_0 \rvert$, $\lvert z(x_0) \rvert$, $\lambda$, $\Lambda$ and $h_\mu(\Omega_r)$ --- the modulus of (H\"older) continuity of the solution over $\Omega_r$, such that for $B_\rho= B_\rho(x_0)$,
     \[
       \int_{B_\rho}V_a^{k_a}\xi_a^2dx\leq\delta\int_{B_\rho}V_a^{k_a-2}\lvert \nabla\xi_a\rvert^2+\int_{B_\rho}V_2^{k_2-1}\xi_1^2,\quad\forall\xi_a\in L_{1,0}^{k_a}\cap C^0(B_\rho).
     \]
   \end{claim}
   In fact, by the assumption of \autoref{thm:Morrey} and note that $z\in C^\mu(\Omega)$, if $\rho\leq \max\left\{ 4r,1 \right\}$, then we have for almost all $x\in B_{\rho}$, $\lvert x \rvert^2+\lvert z(x) \rvert^2\leq R_1^2$, where $R_1>0$ is a constant depending on $\lvert x_0 \rvert$, $\lvert z(x_0) \rvert$ and $h_\mu(\Omega_r)$, and the natural structure condition \eqref{eq:coupled-condi} holds for $R=R_1$ over $B_\rho$. Let $\zeta^i_a(x)\mathpunct{:}=\xi^2_a(x)\left( z_{a}^i(x)-z_{a}^i(x_0) \right)\in L_{1,0}^{k_a}\bigcap C^0(B_\rho)$ be the test function in~\eqref{eq:coupled-weak}, we have
   \[
     \int_{B_\rho}2\xi_a\partial_\alpha\xi_aq_{ai}^\alpha(z_{a}^i-z_{a}^i(x_0))+\xi_a^2\left( p_{a\alpha}^iq_{ai}^\alpha+\left( z_{a}^i-z_{a}^i(x_0) \right)w_{ai} \right)=0.
   \]
   The condition given in~\eqref{eq:coupled-condi} implies that
   \begin{align*}
     \xi_a^2p_{a\alpha}^iq_{ai}^\alpha(x,z,p)&=\xi_a^2p_{a\alpha}^iq_{ai}^\alpha(x,z,0)+\xi_a^2p_{a\alpha}^i\int_0^1 \frac{\partial q_{ai}^\alpha(x,z,tp)}{\partial p_{b\beta}^j}p_{b\beta}^jdt\\
                                             &\geq \xi_a^2p_{a\alpha}^iq_{ai}^\alpha(x,z,0)+\lambda\int_0^1\xi_a^2\lvert p_a\rvert^2(1+|tp_a|^2)^{k_a/2-1}dt\\
                                             &\geq \lambda\int_0^1\xi_a^2\lvert p_a\rvert^2\lvert tp_a\rvert^{k_a-2}-\xi_a^2\lvert p_a\rvert\lvert q_{ai}^\alpha(x,z,0)\rvert\\
                                             &\geq \lambda \xi_a^2\frac{1}{k_a-1}\lvert p_a\rvert^{k_a}-\Lambda\xi_a^2\lvert p_a\rvert.
   \end{align*}
   Since
   \begin{gather*}
     V_a^{k_a}=(1+\lvert p_a\rvert^2)^{k_a/2}\leq 2^{k_a/2-1}\left( 1+\lvert p_a\rvert^{k_a} \right)
     \intertext{and}
     \Lambda\lvert p_a\rvert=\epsilon^{-1/k_a}\Lambda\cdot\epsilon^{1/k_a}\lvert p_a\rvert\leq\epsilon \frac{\lvert p_a\rvert^{k_a}}{k}+\epsilon^{-k_a^*/k_a}\frac{\Lambda^{k_a^*}}{k_a^*},
   \end{gather*}
   we conclude that, for some constants $\lambda', \Lambda'$ depending on $\lambda=\lambda(R_1),\Lambda=\Lambda(R_1)$,
   \[
     \xi_a^2p_{a\alpha}^iq_{ai}^\alpha(x,z,p)\geq \lambda'\xi_a^2V_a^{k_a}-\Lambda'\xi_a^2.
   \]
   Also, from~\eqref{eq:coupled-condi}
   \[
     \lvert q_a \rvert\leq\Lambda V_a^{k_a-1},\forall\,a\in\left\{ 1,2 \right\},\quad
     \lvert w_{1} \rvert\leq\Lambda\left( V_1^{k_1}+V_2^{k_2-1} \right),\quad
     \lvert w_{2}\rvert\leq \Lambda V_2^{k_2}.
   \]
   It follows that
   \begin{align*}
     \int_{B_\rho}\xi_a^2V_a^{k_a}&\leq C(\Lambda,\lambda)\int_{B_\rho}\xi_a^2-\int_{B_\rho}\xi_a\left( 2\partial_\alpha\xi_aq_{ai}^\alpha+\xi_aw_{ai} \right)\left( z_a^i-z_a^i(x_0) \right)\\
     %&\leq C(\Lambda,\lambda)\int_{B_\rho}\xi_a^2+\sup_{B_\rho}\lvert z-z(x_0)\rvert \left( V_a^{k_a/2}\xi_a\cdot V_a^{k_a/2-1}\partial_\alpha\xi_a +  V_a^{k_a}\xi_a^2 \right)+V_2^{k_2-1}\xi_1^2\\
      &\leq C(\Lambda,\lambda)\int_{B_\rho}\left[ \xi_a^2+\sup_{B_\rho}\lvert z-z(x_0) \rvert\left( \left( V_a^{k_a-2}\lvert \nabla\xi_a \rvert^2+V_a^{k_a}\xi_a^2 \right)+V_2^{k_2-1}\xi_1^2 \right) \right].
   \end{align*}
   Now, the Poincar\'e inequality implies that (note that $\xi_a\in L_{1,0}^{k_a}(B_\rho)$),
   \[
     \int_{B_\rho}\xi_a^2\leq C\rho^2\int_{B_\rho}\lvert \nabla\xi_a \rvert^2\leq C\rho^2\int_{B_\rho}\lvert \nabla\xi_a \rvert^2V_a^{k_a-2}.
   \]
   The claim follows from the fact that $z_a\in C^\mu(\bar B_\rho)$ and $\sup_{B_\rho}\lvert z-z(x_0) \rvert$ can be chosen as small as we need, provided that $\rho$ is small enough.

   To apply the above claim, we take $r=\rho/4$ further small, where $\rho$ is the constant in the above claim, $D'=B_{3r}=B_{3r}(x_0)$, $D=B_r=B_r(x_0)$, clearly $B_{4r}=B_{4r}(x_0)\subset\!\subset\Omega$, and $\tilde{z}_a \mathpunct{:}=z(\cdot+he_\gamma)\in C^\mu(\bar B_{3r})$ for any $0<\lvert h \rvert<r$. Moreover, $\tilde{z}=(\tilde{z}_1,\tilde{z}_2)$ solves~\eqref{eq:coupled-weak} with $\tilde{q} \mathpunct{:}=q(x+he_\gamma,\cdot,\cdot)$ and $\tilde{w} \mathpunct{:}=w(x+he_\gamma,\cdot,\cdot)$; and as coefficients they satisfy the condition~\eqref{eq:coupled-condi} on $B_{3r}$ with $R=R_1$. Thus, we can apply the above claim in $B_{3r}$ for $\tilde{z}$ to obtain (note that $4r\leq\rho$)
   \[
     \int_{B_{3r}}\tilde{V}_a^{k_a}\xi_a^2\leq\delta\int_{B_{3r}}\tilde{V}_a^{k_a-2}\lvert \nabla\xi_a \rvert^2+\int_{B_{3r}}\tilde{V}_2^{k_2-1}\xi_1^2,\quad\forall\xi_a\in L_{1,0}^{k_a}\cap C^0(B_{3r}),
   \]
   where
   \[
     \tilde{V}_a^2=1+\lvert p_a(x+he_\gamma) \rvert^2=1+\lvert p_a+\Delta p_a \rvert^2.
   \]
   Since $\mathrm{supp}\,\eta\subset\!\subset D_r'=B_{2r}$ and $z_{ah}\in L_1^{k_a}\cap C^0(B_{3r})$, we can take $\xi_a=Z_a=\eta z_{ah}$ to obtain
   \[
     \int_{B_{3r}}\tilde{V}_a^{k_a}\lvert Z_a \rvert^2\leq\delta\int_{B_{3r}}\tilde{V}_a^{k_a-2}\lvert \nabla Z_a \rvert^2+\int_{B_{3r}}\tilde{V}_2^{k_2-1}\lvert Z_1 \rvert^2.
   \]
   Clearly,
   \[
     \int_{B_{3r}} V_a^{k_a}\lvert Z_a \rvert^2\leq\delta\int_{B_{3r}} V_a^{k_a-2}\lvert \nabla Z_a \rvert^2+ \int_{B_{3r}}V_2^{k_2-1}\lvert Z_1 \rvert^2.
   \]
   Now, we can estimate
   \begin{align*}
     \int_{B_{3r}}A_{ah}Q_{ah}\lvert Z_a \rvert^2&=\int_{B_{3r}}\int_{0}^1\left( 1+\lvert p_a+t\Delta p_a \rvert^2 \right)^{k_a/2}\lvert Z_a \rvert^2dt\\
                                                 &\leq C\int_{B_{3r}}\left( \tilde{V}_a^{k_a}+V_a^{k_a} \right)\lvert Z_a \rvert^2\\
                                                 &\leq C\delta\int_{B_{3r}} \left( \tilde{V}_a^{k_a-2}+V_a^{k_a-2} \right)\lvert \nabla Z_a \rvert^2+C\int_{B_{3r}}\left( V_2^{k_2-1} + \tilde{V}_2^{k_2-1} \right)\lvert Z_1 \rvert^2 \\
                                                 &\leq \frac{C\delta}{c}\int_{B_{3r}} \int_0^1\left( 1+\lvert p_a+t\Delta p_a \rvert^2 \right)^{k_a/2-1}\lvert \nabla Z_a \rvert^2+C\int_{B_{3r}}\left( V_2^{k_2-1} + \tilde{V}_2^{k_2-1} \right)\lvert Z_1 \rvert^2 \\
                                                 &=\frac{C\delta}{c}\int_{B_{3r}} A_{ah}\lvert \nabla Z_a \rvert^2+\left( V_2^{k_2-1} + C\int_{B_{3r}}\tilde{V}_2^{k_2-1} \right)\lvert Z_1 \rvert^2 ,
   \end{align*}
   where in the second and fourth lines, we used the following elementary inequalities (see~\cite{Morrey2008Multiple}*{p.~189, (5.9.4)}). For $q=k_a/2$ or $q=k_a/2-1$, there exist some constants $c$, $C$ such that
   \begin{multline*}
     c\left((1+|p_a|^2)^q+(1+|p_a+\Delta p_a|^2)^q\right)\\
     \leq \int_0^1(1+|p_a+t\Delta p_a|^2)^qdt\\
     \leq C\left((1+|p_a|^2)^q+(1+|p_a+\Delta p_a|^2)^q\right).
   \end{multline*}
   Thus, by~\eqref{eq:AQZ2}, if we take $\delta$ small enough
   \begin{equation}\label{eq:l22-estimate}
     \int_{B_r}A_{ah}\lvert \nabla z_{ah} \rvert^2\leq C(\Lambda,\lambda)\int_{B_{3r}}\left[ A_{ah}\left( \frac{\lvert z_{ah} \rvert^2}{r^2}+Q_{ah} \right)+\left( A_{2h}P_{2h}+V_2^{k_2-1}+\tilde{V}_2^{k_2-1} \right)\lvert z_{1h} \rvert^2 \right].
   \end{equation}
   Since $p_{a\gamma}=\partial_\gamma z_{a}\in L^{k_a}(B_{4r})$ by assumption, the relation of weak derivatives and differential quotients implies $z_{ah} \mathpunct{:}=\Delta_\gamma^h z_a\in L^{k_a}(B_{3r})$ too and $z_{ah}\to\partial_\gamma z_{a}$ in $L^{k_a}(B_{3r})$. Since $V_a(t)=(1+|p_a+t\Delta p_a|^2)^{1/2}\in L^{k_a}(B_{3r})$, we know that $V_a^{k_a-2}(t)\in L^{k_a/(k_a-2)}(B_{3r})$ and $A_{ah}=\int_0^1 V_a^{k_a-2}(t)dt\to A_a=V_a^{k_a-2}$ in $L^{k_a/(k_a-2)}(B_{3r})$ by~\cite{Morrey2008Multiple}*{Thm.~3.6.8}. A similar argument shows that $A_{ah}P_{ah}=\int_0^1V_a^{k_a-1}(t)dt\to V_a^{k_a-1}$ in $L^{k_a/(k_a-1)}(B_{3r})$ and $A_{ah}Q_{ah}=\int_0^1 V_a^{k_a}(t)dt\to V_a^{k_a}$ in $L^1(B_{3r})$. Applying H\"older's inequality, we know that the right-hand side of~\eqref{eq:l22-estimate} is uniformly bounded (independent of $h$). Here, we need the additional assumption $z_1\in L_1^{2k_2}(B_{4r})$ to conclude that the terms in the second parentheses of~\eqref{eq:l22-estimate} are uniformly bounded. Since $A_{ah}\geq1$, we conclude from~\eqref{eq:l22-estimate} that $\lVert \nabla z_{ah} \rVert_{L^2(B_r)}$ is uniformly bounded. But $z_{ah}\in L^{k_a}(B_r)$ with $k_a\geq2$ thus $z_{ah}$ is uniformly bounded in $L_1^2(B_r)$. The weak compactness implies $z_{ah} \weakto v_a$ in $L_1^2(B_r)$ for some sub-sequence $h\to0$. The compact embedding $L_1^2(B_r)\to L^{k_a}(B_r)$ implies that, after taking a further sub-sequence, $z_{ah}\to v_a$ in $L^{k_a}(B_r)$, but we already shown $z_{ah}\to \partial_\gamma z_{a}\in L^{k_a}(B_r)$ ($k_a\geq2$), thus $\partial_\gamma z_{a}=v_a\in L^2_1(B_r)$. Since $\gamma$ is arbitrary, it shows that $z_a\in L_2^2(B_r)$ and the proof is completed by the arbitrariness of $x_0$.
 \end{proof}
 \subsection{The smoothness of perturbed solution}\label{sec:smoothness}
 We first write down the Euler--Lagrange equation of $\mathcal{L}_\alpha$ locally in terms of Fermi coordinates, then the $L_2^2$-interior regularity follows directly from \autoref{thm:Morrey}. To prove the boundary regularity, we  extend the solution from half disc to the whole disc via a reflection argument. It turns out that such reflected solution satisfies an equation that is similar to the original one (with coefficients extended properly), c.f.~\autoref{lem:bdry-reg}. The verification of this fact is given by decomposing the test function through parity and check the parity of each coefficient. It is notable that in general the coefficient involving the Christoffel symbols of extended solution is only $L^\infty$ near the free boundary, and we cannot apply \autoref{thm:Morrey} directly to the extended solution to show the boundary regularity. This explains the additional requirement that $K\subset F$ is totally geodesic. Finally, the smoothness up to the boundary of critical points of $\mathcal{L}_\alpha$ for $\alpha-1$ small follows from a bootstrap of the $L_2^2$-strong solution.

 Locally, we take coordinate systems near the boundary as $\left\{ U;x=(x^1,x^2) \right\}$ with $\partial\Sigma\cap U=\set{x^2=0}$ and for any $(x^1,0)\in\partial\Sigma\cap U$, let $x^2\mapsto(x^1,x^2)$ be a regular geodesic orthogonal to $\partial\Sigma$. Let $B=\left\{ x\in\Sigma:|x|<1\right\}$ be the unit disc in $\Sigma$, $D=\left\{ x\in\Sigma:|x|<1,x^2\geq0 \right\}$ be the unit upper half disc in $\Sigma$, $\partial^0 D=\left\{ x\in\partial D:x^2=0 \right\}$ and $\partial^+D=\left\{ x\in\partial D:|x|=1 \right\}$. For simplicity, we use $U$ to denote either $B$ or $D$.
 The following theorem implies that locally we can always choose a representative that is in Coulomb gauge.
 \begin{thm}[\citelist{\cite{Uhlenbeck1982Connections}*{Thm.~2.1}\cite{Marini1992Dirichlet}*{Thm.~3.2$'$--3.3$'$}}]\label{thm:coulomb}
   Suppose $p\geq1$, $G$ is a compact Lie group and $\mathcal{U}\mathpunct{:}=U\times G$ is the trivial bundle on a disc/half disc $U\subset \mathbb{R}^2$ with flat metric on $U$. Then, there exists a uniform constant $\delta_0>0$, such that any connection $\tilde{A}\in L_1^p(\Omega^1(\mathcal{U}\times_{\mathrm{Ad}}\mathfrak{g}))$ with $\lVert F_{\tilde{A}}\rVert_{L^1(\Omega^2(\mathcal{U}\times_{\mathrm{Ad}}\mathfrak{g}))}\leq\delta_0$ is gauge equivalent to a connection $d+A\in L_1^p(\Omega^1(\mathcal{U}\times_{\mathrm{Ad}}\mathfrak{g}))$, i.e., for some $S\in L_2^p(\mathcal{U}\times_cG)$, $S^*\tilde{A}=d+A$, where $A$ satisfies
   \begin{enumerate}
     \item $d^*A=0$, where $*$ is the Hodge star operator with respect to the flat metric;
     \item $\nu\lh A=0$ for any $x\in\partial U$;
     \item\label{item:coulomb-est} $\lVert A\rVert_{L_1^p}\leq C\lVert F_A\rVert_{L^p}$.
   \end{enumerate}
 \end{thm}
 Suppose $\sigma\mathpunct{:}U\times F\to\pi^{-1}(U)$ is a local trivialization of $\mathcal{F}$. Under this trivialization, we write the section $\phi(x)=(x,u(x))\in U\times F$ and identify $\phi$ with $u$, $\nabla_A\phi$ with $\nabla_Au$ and $\mu(\phi)$ with $\mu(u)$, since their values are determined by $u$. With these notations, when the metric on $U$ is Euclidean and $u$ is regular enough, we can rewrite~\eqref{eq:EL-L-alpha-global} as (under Coulomb gauge)
 \begin{equation}\label{eq:EL-L-alpha-local}
   \begin{cases}
     \nabla_A^*\nabla_Au-  \frac{1}{\Upsilon}\left\langle d\Upsilon,\nabla_Au \right\rangle-\frac{1}{\Upsilon}\mu(u)\cdot\nabla\mu(u)=0,&x\in U\\
     \Delta A-\left\langle dA,A \right\rangle-\left\langle A,[A,A] \right\rangle+\Upsilon\left\langle \nabla_Au,u \right\rangle=0,&x\in U\\
     \nu\lh \nabla_A\phi\perp T_\phi\mathcal{K}^v,&x\in\partial^0U\\
     \nu\lh F_A=0,&x\in\partial^0U\\
     \nu\lh A=0,&x\in\partial U,
   \end{cases}
 \end{equation}
 where $\partial^0 U \mathpunct{:}=\partial\Sigma\cap U$, $\Upsilon=\alpha(1+\lvert \nabla_Au \rvert^2)^{\alpha-1}$, $\Delta A=d^*dA+dd^*A$ is the Laplace operator on 1-forms, and we use
 \[
   \nabla_A^*(f\nabla_A\phi)=-\left\langle df,\nabla_A\phi \right\rangle+f\nabla_A^*\nabla_A\phi.
 \]
 Note that, by definition $\nabla_Au=du+A\cdot u=(\partial_iu+A_i\cdot u)dx^i\mathpunct{:}=u_{|i}dx^i$, where $A_i\in \mathfrak{g}$, which acts on $u$ as follows
 \begin{equation}\label{eq:Au}
   A_i\cdot u \mathpunct{:}=\left. \frac{d}{dt} \right\rvert_{t=0}\exp(tA_i)\cdot u,
 \end{equation}
 here, $\exp$ is the exponential map of $G$. It is clear that $\nabla_Au$ is a tangent vector of $F$ at $u$, we will write $A_i\cdot u \mathpunct{:}=A_i^\sharp(u)$ be the \emph{fundamental vector field} corresponding to $A_i$ at $u$. Similarly, for a tangent vector fields $v\in\Gamma(u^*TF)$, we have $\nabla_Av=\nabla v+A\cdot v=(\nabla_{\partial_i}v+A_i\cdot v)dx^i$, where $\nabla_{\partial_i}$ is the pullback connection, and
 \begin{equation}\label{eq:Av}
   A_i\cdot v \mathpunct{:}=\left. \frac{\nabla}{dt} \right\rvert_{t=0}\left[ \left( d(\exp(tA_i)\cdot u) \right)(v) \right]=\nabla_v A_i^\sharp,
 \end{equation}
 where $\nabla$ is the Levi-Civita connection of $F$. Suppose $\mathfrak{g}=\Span{v_1,\ldots,v_m}$, and denote $V_\alpha$ the fundamental vector field generated by $v_\alpha$, then for $A_i(x)=a_i^\alpha(x)v_\alpha$, we have
 \[
   A_i\cdot u=A_i^\sharp(u)=a_i^\alpha(x)V_\alpha(u).
 \]

 Now, a direct computation shows that the local equation is given by
 \begin{equation}\label{eq:A-u-strong}
   \begin{cases}
     \Delta_\Sigma u-2(\alpha-1)\frac{\left\langle \nabla_A^2u,\nabla_Au \right\rangle \nabla_Au}{1+|\nabla_Au|^2}-\Phi_\alpha(A,u)=0,&x\in U\\
     \Delta A-\Psi_\alpha(A,u)=0,&x\in U\\
     \frac{\partial u}{\partial\nu}\perp T_uK,&x\in\partial^0 U\\
     A_2=0,&x\in\partial U\\
     \frac{\partial A_1}{\partial\nu}=0,&x\in\partial^0 U,
   \end{cases}
 \end{equation}
 where
 \begin{align}
   \Phi_\alpha(A,u)&=\Gamma(u)(du,du)+2A\cdot du+A\cdot A\cdot u+\frac{1}{\Upsilon}\mu(u)\cdot\nabla\mu(u),\label{eq:Phi}\\
   \intertext{$\Delta_\Sigma$ is the Laplace-Beltrami operator on functions over $\Sigma$, $\Delta=dd^*+d^*d$ is the Laplace operator of 1-forms, $\Gamma(u)$ is the second fundamental form of $F \hookrightarrow \mathbb{R}^l$, and}
   \Psi_\alpha(A,u)&=\left\langle dA,A \right\rangle+\left\langle A,[A,A] \right\rangle-\Upsilon\left\langle \nabla_Au,u \right\rangle\label{eq:Psi}.
 \end{align}
 The boundary condition is localized as follows: let $\left\{ e_1,e_2 \right\}$, $e_2|_{\partial\Sigma}=\nu$, be a moving frame near the boundary and $\left\{ \omega^1,\omega^2 \right\}$ be the dual frame. If we write $A=A_i\omega^i$, then
 \[
   \begin{cases}
     \nu\lh A=0,&x\in\partial U\\
     \nu\lh F_A=0,&x\in\partial^0 U
   \end{cases}\quad\xtext{implies}\quad
   \begin{cases}
     A_2=0,&x\in\partial U\\
     \frac{\partial A_1}{\partial\nu}=0,&x\in\partial^0 U.
   \end{cases}
 \]
 The boundary condition for the section $\phi$ is given by
 \[
   \nu\lh \nabla_A\phi\perp T_\phi\mathcal{K}^v,\quad x\in\partial^0 U
 \]
 which is equivalent to
 \[
   \frac{\partial u}{\partial \nu}\perp T_uK,\quad x\in\partial^0 U,
 \]
 since $\nu\lh A=0$ on $\partial^0U$.
 \begin{rmk}
   The boundary condition imposed on $\partial^0U$ in~\eqref{eq:A-u-strong} is empty if $U$ is an interior neighborhood. For the boundary neighborhood, the free-boundary is only prescribed at the flat part $\partial^0U$. We should remark also that $A_2=0$ is exactly the local Coulomb gauge boundary condition given by $\nu\lh A=0$ as in~\autoref{thm:coulomb}.
 \end{rmk}
 Before we get involved into the proof of~\autoref{thm:alpha-smoothness}, we state the reflecting technique as follows, which will be needed in the proof of boundary regularity. For simplicity, we will assume that the underlying metric on a local chart $U$ is flat in the following context. Recall that the metric of a two dimensional surface $\Sigma$ is locally conformal to the standard Euclidean metric, i.e., $g=e^{2v}g_0$, where $g_0$ is the Euclidean metric. Then, the YMH energy has the form
 \[\begin{aligned}
     \mathcal{L}_\alpha(A,u)&=\int\Big((1+|\nabla_Au|^2_g)^\alpha + |F_A|_g^2 +|\mu(u)|_h^2\Big) dv_g\\
                            &=\int\Big((1+e^{-2v}|\nabla_Au|^2_{g_0})^\alpha + e^{-4v}|F_A|_{g_0}^2 +|\mu(u)|_h^2\Big) e^{2v}dv_{g_0}\\
                            &=\int\Big( e^{-2(\alpha-1)v}(e^{2v}+|\nabla_Au|^2_{g_0})^\alpha + e^{-2v}|F_A|_{g_0}^2 +e^{2v}|\mu(u)|_h^2\Big)dv_{g_0},
 \end{aligned}\]
 thus the Euler-Lagrangian equation under the conformal metric $g_0$ is given by
 \[
   \begin{dcases}
     \nabla_A^*\Big(\alpha e^{-2(\alpha-1)v}(e^{2v}+|\nabla_Au|^2_{g_0})^{\alpha-1}\nabla_A u\Big)+e^{2v}\mu(u)\nabla\mu(u)=0,\\
     D_A^*\Big(e^{-2v}F_A\Big)+\alpha e^{-2(\alpha-1)v}(e^{2v}+|\nabla_Au|^2_{g_0})^{\alpha-1}\langle\nabla_A u, u\rangle=0.
   \end{dcases}
 \]
 Although the YMH field equation is not conformally invariant, it is clear from equation (\ref{eq:EL-L-alpha-global}) and (\ref{eq:A-u-strong}) that after a conformal change of the metric, the structure of the equation will not change, except that some additional lower order terms (which can be analytically well controlled) emerge. Therefore our arguments blow  still works for non-flat metrics. Especially, in the proofs of \autoref{thm:alpha-smoothness} and \autoref{thm:blowup}, we only consider the case that the metric is locally flat for simplicity and omit the general case with the conformal factor $v$.

 Now for $x_0\in\partial\Sigma$, without loss of generality, we assume the local trivialization chart $U$ of $x_0$ is an upper half disc $D_\rho$ centered at $x_0=0$ and the flat boundary is settled on $\partial\Sigma$. Moreover, since $u\in L_1^{2\alpha}(\Sigma,\mathbb{R}^l) \hookrightarrow C^0(\bar\Sigma,\mathbb{R}^l)$, we can take $\rho$ small enough such that the following reflection is well-defined. A more geometric way can be found in~\cite{Scheven2006Partial}*{Sect.~3}. For $p=u(x_0)\in K$, we choose Fermi coordinates $\left( f^1,\ldots,f^n \right)$ on an open neighborhood $V$ of $p$ in $F$, such that
 \begin{itemize}
   \item $V\cap K=\left\{ f^{k+1}=0, \ldots, f^n=0 \right\}$;
   \item For any fixed $q\in K$ and $a\in\left\{ k+1,\ldots, n \right\}$, the $f^a$-coordinate curve start from $q$ is a geodesic in $V\subset F$, which is perpendicular to $K$.
 \end{itemize}
 In order to keep the extension as smooth as possible, it turns out that the extension depends on the ``type'' of boundary condition. More precisely, for homogeneous Neumann boundary we use the even extension and for Dirichlet boundary, we use the odd extension. These two types of boundary conditions root in the free boundary condition, the $n-k$ Dirichlet conditions come from the fact that $u(\partial^0 U)\subset K$. The remaining $k$ boundary conditions come from the constraint in calculus of variation. To write down these boundary conditions in Fermi coordinates, we note first that
 \[
   \frac{\partial u}{\partial\nu}=-\frac{\partial u}{\partial x^2}=-\frac{\partial u^a}{\partial x^2}\frac{\partial}{\partial f^a},
 \]
 where $u^a \mathpunct{:}=f^a\circ u$. Then, the local boundary condition in~\eqref{eq:A-u-strong} of $u$ is given by
 \begin{equation}\label{eq:u-Neumann-bdry}
   \frac{\partial u^a}{\partial x^2}(x)=0,\quad x\in\partial^0U, \quad a\in\I_1 \mathpunct{:}=\left\{ 1,\ldots,k \right\}.
 \end{equation}
 The constraint $u(\partial^0 U)\subset K$ transforms to
 \begin{equation}\label{eq:u-Dirichlet-bdry}
   u^a(x)=0,\quad x\in\partial^0U,\quad a\in \I_2 \mathpunct{:}=\left\{ k+1,\ldots,n \right\}.
 \end{equation}

 Next, we extend various quantities from $D_\rho$ to $B_\rho$. Let us illustrate the basic idea by the extension of $u$. Suppose $x^*=(x^1,-x^2)\in D_\rho$ is the reflection of $x=(x^1,x^2)$ respect to $\partial^0D_\rho$ and $r(x)=x^*$ is the reflection map. The \emph{reflection} of $F$ with respect to $K$ is defined as follows,
 \begin{align*}
   \gamma \mathpunct{:}F|_V&\to F\\
   q=\varphi^{-1}(f^1,\ldots, f^n)&\mapsto q^*=\varphi^{-1}(f^1,\ldots, f^k,-f^{k+1},\ldots, -f^n),
 \end{align*}
 where $\varphi \mathpunct{:}F|_V\to \mathbb{R}^n$ is the coordinate map. The extension of $u$ is given by, for $x\in D_\rho^{-}\mathpunct{:}=B_\rho\setminus D_\rho$\footnote{In what follows, we always omit the trivial relation that the extended quantity restricting to $D_\rho$ equals to the original one for simplicity.},
 \[
   \tilde{u}=\gamma\circ u\circ r.
 \]
 In order to check the parity, we note that, for $x\in D_\rho^-$,
 \[
   \tilde{u}^a(x)=
   \begin{cases}
     u^a(x^*), & a\in\I_1\\
     -u^a(x^*),& a\in\I_2.
   \end{cases}
 \]
 By our boundary conditions \eqref{eq:u-Neumann-bdry} and \eqref{eq:u-Dirichlet-bdry} of $u$, it is easy to see, $\tilde{u}\in L_1^{2\alpha}(B_\rho,F)$ for any $u\in L_1^{2\alpha}(D_\rho,F)$.

 For a vector field $v(x)=v^a(x)\partial_{f^a}(u(x))$ along $u$, we define the extended vector filed along $\tilde{u}$ as, for $x\in D_\rho^-$,
 \[
   \tilde{v}\mathpunct{:}=\gamma_*\circ v\circ r.
 \]
 In particular, for $x\in D_\rho^-$,
 \[
   \widetilde{\partial_{f^a}(u)}=\gamma_*(\partial_{f^a}(u(x^*)))=(-1)^{j-1}\partial_{f^a}(\tilde{u}(x)),\quad a\in\I_j,
 \]
 and if we write $\tilde{v}(x)=\tilde{v}^a(x)\partial_{f^a}(\tilde{u}(x))$, then
 \[
   \tilde{v}^a(x)=(-1)^{j-1}v^a(x^*),\quad a\in\I_j.
 \]

 The metric $h$ is extended by $\tilde{h}=\gamma^*h$. It is easy to show, in the coordinates $(f^1,\ldots, f^n)$, for $x\in D_\rho^-$,
 \[
   \tilde{h}_{ab}(\tilde{u}(x))=
   \begin{cases}
     -h_{ab}(u(x^*)),&\xtext{if}(a,b)\in-\I\Lambda \mathpunct{:}=\I_1\times\I_2\bigcup\I_2\times\I_1\\
     h_{ab}(u(x^*)),&\xtext{otherwise}.
   \end{cases}
 \]
 The extended Christoffel symbol $\tilde{\Gamma}(\tilde{u})$\footnote{It is easy to check, for $(a,b,c)\in\I_1\times\I_1\times\I_1\cup\I_2\times\I_2\times\I_1\cup\I_2\times\I_1\times\I_2\cup\I_1\times\I_2\times\I_2$ and $x\in D_\rho^-$, $\tilde{\Gamma}_{ab}^c(\tilde{u}(x))=\Gamma_{ab}^c(u(x^*))$ and $\tilde{\Gamma}_{ab}^c(\tilde{u}(x))=-\Gamma_{ab}^c(u(x^*))$ otherwise.} is defined by the extended metric $\tilde{h}(\tilde{u})$.

 We extend the connection one form $A$ from $D_\rho$ to the whole disc $B_\rho$ \emph{evenly}, i.e., we define $\tilde{A}$ by the following relation,
 \[
   r^*\tilde{A}=\tilde{A}.
 \]
 If we write $A$ as a $\mathfrak{g}$-valued 1-form $A_i(x)dx^i$ locally, then for $x\in D_\rho^-$,
 \[
   \tilde{A}_1(x)=A_1(x^*),\quad \tilde{A}_2(x)=-A_2(x^*).
 \]
 Locally, let $\mathfrak{g}=\mathrm{span}\left\{ v_1,\ldots, v_m \right\}$, the fundamental vector field generated by $v_\alpha$ is denoted by $V_\alpha=\lambda_\alpha^a\partial_{f^a}$. If we extend $V_\alpha$ by, for $x\in D_\rho^-$,
 \[
   \widetilde{V}_\alpha(\tilde{u})\mathpunct{:}=\tilde{\lambda}_\alpha^a(\tilde{u})\partial_{f^a}(\tilde{u})=\gamma_*\circ V_\alpha(u)\circ r,
 \]
 then
 \begin{equation}\label{eq:lambda}
   \tilde{\lambda}_\alpha^a(\tilde{u}(x))=(-1)^{j-1}\lambda_\alpha^a(u(x^*)),\quad a\in \I_j.
 \end{equation}
 If we write $A_i(x)=a_i^\alpha(x)v_\alpha$ and $\tilde{A}_i(x)=\tilde{a}_i^\alpha(x)v_\alpha$, then for $x\in D_\rho^-$,
 \begin{equation}\label{eq:a-i-alpha}
   \tilde{a}_i^\alpha(x)=(-1)^{i-1}a_i^\alpha(x^*).
 \end{equation}
 Note that the local boundary condition of $A$ is given by, for any $\alpha=1,2,\ldots,m$ and any $x\in\partial^0U$,
 \[
   \frac{\partial a_1^\alpha}{\partial x^2}(x)=0,\quad a_2^\alpha(x)=0,
 \]
 which clearly implies that $\tilde A\in L_1^2(B, \Omega^1(\mathfrak{g}))$ for any $A\in L_1^2(B, \Omega^1(\mathfrak{g}))$.

 Since we expect $\widetilde{\nabla}_{\tilde A}\tilde u$ and $\widetilde{\nabla}_{\tilde{A}}\tilde{v}$ are vector fields along $\tilde u$, we extend them, respectively, as follows, for $x\in D_\rho^-$,
 \[
   \widetilde{\nabla}_{\tilde{A}}\tilde{u}(x)=\widetilde{\nabla_Au}(x)=\gamma_*(\nabla_Au(x^*)), \quad
   \widetilde{\nabla}_{\tilde{A}}\tilde{v}(x)=\widetilde{\nabla_{A}v}(x)=\gamma_*(\nabla_Av(x^*)).
 \]
 Locally, if we write
 \begin{align*}
   \widetilde{\nabla}_{\tilde{A}}\tilde{u}(x)&=\left( \widetilde{\nabla}_{\partial_i}\tilde{u}(x)+\tilde{A}_i\tilde{\cdot}\tilde{u}(x) \right)\otimes dx^i\\
   \widetilde{\nabla}_{\tilde{A}}\tilde{v}(x)&=\left( \widetilde{\nabla}_{\partial_i}\tilde{v}(x)+\tilde{A}_i\tilde{\cdot}\tilde{v}(x) \right)\otimes dx^i,
 \end{align*}
 where $\widetilde{\nabla}_{\partial_i}$ is the pullback connection along $\tilde{u}$, then the above extension requires that $\tilde{A}_i\tilde{\cdot}\tilde{u}$ and $\tilde{A}_i\tilde{\cdot}\tilde{v}$ satisfy the following relations, respectively, for $x\in D_\rho^-$:
 \begin{align}
   \tilde{A}_i\tilde{\cdot}\tilde{u}(x)&=(-1)^{i-1}\widetilde{A_i\cdot u}(x)=(-1)^{i-1}\gamma_*\left( A_i\cdot u(x^*) \right),\label{eq:Aiu}\\
   \tilde{A}_i\tilde{\cdot}\tilde{v}(x)&=(-1)^{i-1}\widetilde{A_i\cdot v}(x)=(-1)^{i-1}\gamma_*\left( A_i\cdot v(x^*) \right)\label{eq:Aiv}.
 \end{align}
 Recall that
 \[
   A_i\cdot u=A_i^\sharp(u)=a_i^\alpha(x)V_\alpha(u)=a_i^\alpha(x)\lambda_\alpha^a(u)\partial_{f^a}(u)=\mathpunct{:}A_i^a(u)\partial_{f^a}(u),
 \]
 If we write
 \[
   \tilde{A}_i\tilde{\cdot}\tilde{u}%=\tilde{A}_i^\sharp(\tilde{u})
   \mathpunct{:}=\tilde{a}_i^\alpha(x)\tilde{\lambda}_\alpha^a(\tilde{u})\partial_{f^a}(\tilde{u})
   =\mathpunct{:}\tilde{A}_i^a(\tilde{u})\partial_{f^{a}}(\tilde{u}),
 \]
 where $\tilde{\lambda}_\alpha^a(\tilde{u})=(-1)^{j-1}\lambda_\alpha^a(u\circ r)$ for $x\in D_\rho^-$ and $a\in \I_j$ (by \eqref{eq:lambda}), then by \eqref{eq:Aiu}, for $x\in D_\rho^-$,
 \[
   \tilde{A}_i^a(\tilde{u}(x))=(-1)^{i+j}A_i^a(u(x^*)), \quad a\in\I_j.
 \]
 To simplify the notation further, let $\tilde{u}^a_{|i}(x)\mathpunct{:}=\partial_i\tilde{u}^a(x)+\tilde{A}_i^a(\tilde{u}(x))$, then for $x\in D_\rho^-$,
 \[
   \tilde{u}^a_{|i}(x)=(-1)^{i+j}u^a_{|i}(x^*),\quad a\in\I_j,
 \]
 where
 \begin{equation}\label{eq:nabla-A-u}
   u^a_{|i}(x^*)\mathpunct{:}=\partial_iu^a(x^*)+A_i^a(u(x^*))\text{ and } \nabla_Au(x^*)=u_{|i}^a(x^*)\partial_{f^a}(u(x^*))\otimes dx^i.
 \end{equation}

 In order to show the local expression of $\widetilde{\nabla}_{\tilde{A}}\tilde{v}$, recall that
 \[
   \widetilde{\nabla}_{\tilde{A}}\tilde{v}(x)=\left( \widetilde{\nabla}_{\partial_i}\tilde{v}(x)+\tilde{A}_i\tilde{\cdot}\tilde{v}(x) \right)\otimes dx^i,
 \]
 where $\widetilde{\nabla}_{\partial_i}$ is the pullback connection along $\tilde{u}$. More precisely,
 \begin{align*}
   \widetilde{\nabla}_{\partial_i}\tilde{v}(x)&=\widetilde{\nabla}_{\partial_i}\left( \tilde{v}^a(x)\partial_{f^a}(\tilde{u}(x)) \right)\\
                                              &=\partial_i\tilde{v}^a(x)\partial_{f^a}(\tilde{u}(x))+\tilde{v}^a(x)\widetilde{\nabla}_{\tilde{u}_*\partial_i}\partial_{f^a}(\tilde{u}(x))\\
                                              &=\left( \partial_i\tilde{v}^c(x)+\tilde{v}^a(x)\partial_i\tilde{u}^b(x)\widetilde{\Gamma}_{ab}^c(\tilde{u}(x)) \right)\partial_{f^c}(\tilde{u}(x))\\
                                              &=(-1)^{i-1}\widetilde{\nabla_{\partial_i}v}(x),
 \end{align*}
 where
 \begin{equation}\label{eq:nabla-partial_i-v}
   \nabla_{\partial_i}v(x)=\left( \partial_i v^c(x)+ v^a(x)\partial_i u^b(x)\Gamma_{ab}^c( u(x)) \right)\partial_{f^c}( u(x)).
 \end{equation}
 In order to write down $\tilde{A}_i\tilde{\cdot}\tilde{v}(x)=\widetilde{\nabla}_{\tilde{v}}\left( \tilde{A}_i\tilde{\cdot}\tilde{u} \right)(x)$ locally, we adopt the following notation
 \[
   \tilde{A}_i\tilde{\cdot}\partial_{f^a}(\tilde{u})(x)\mathpunct{:}=\tilde{A}_{ai}^b(\tilde{u}(x))\partial_{f^b}(\tilde{u}(x)),
 \]
 then $\tilde A_i\tilde\cdot\tilde v(x)=\tilde v^a(x)\tilde A_{ai}^b(\tilde u(x))\partial_{f^b}(\tilde u(x))$. By \eqref{eq:Aiv}, we know that
 \[
   \tilde A_i\tilde{\cdot}\partial_{f^a}(\tilde u)(x)=(-1)^{i+j}\widetilde{\left( A_i\cdot\partial_{f^a}(u) \right)}(x),\quad a\in\I_j,
 \]
 which implies, for $a\in\I_j,b\in\I_k$ and $x\in D_\rho^-$,
 \[
   \tilde{A}_{ai}^b(\tilde{u}(x))=(-1)^{i+j+k-1}A_{ai}^b(u(x^*)),
 \]
 where
 \begin{equation}\label{eq:A-v}
   \begin{split}
     A_{ai}^b(u(x))&=\partial_{f^a}A_i^b(u(x))+A_i^c(u(x))\Gamma_{ac}^b(u(x))\\
                   &=a_i^\alpha(x)\left( \partial_{f^a}\lambda_\alpha^b(u(x))+ \lambda_\alpha^c(u(x))\Gamma_{ac}^b(u(x)) \right).
   \end{split}
 \end{equation}

 The extension of $\mu$ is given by, for $x\in D_\rho^-$,
 \[
   \tilde{\mu}(\tilde{u}(x))=\mu(u(x^*)).
 \]
 It is easy to show, if we define
 \[
   \widetilde{\nabla}_{\partial f^b}\tilde{\mu}(\tilde{u}(x)) \mathpunct{:}=\widetilde{\left( \nabla_{\partial_{f^b}}\mu(u(x)) \right)}=\gamma_*\left( \left( \nabla_{\partial_{f^b}}\mu \right)(u(x^*)) \right)
 \]
 then, for $b\in\I_j$ and $x\in D_\rho^-$,
 \[
   \widetilde{\nabla}_{\partial_{f^b}}\tilde{\mu}(\tilde{u}(x))=(-1)^{j-1}\nabla_{\partial_{f^b}}\mu(u(x^*)).
 \]
 We should remark that, if we view $\tilde{h}$, $\tilde{\Gamma}$ and $\tilde{\mu}$ as functions of $\tilde{u}$, then they maybe multi-valued, but they are still single-valued when restrict to $\tilde{u}(B_\rho)$ for $\rho$ is small enough, and we can apply \autoref{thm:Morrey} to improve the regularity.

 The following lemma asserts that under the above extension, $(\tilde{A},\tilde{u})$ solves weakly an equation that is similar to~\eqref{eq:EL-L-alpha-local}.
 \begin{lem}\label{lem:bdry-reg}
   Suppose $(\tilde{A},\tilde{u})$ is the extension of $(A,u)$ as above, where $(A,u)\in L_1^2(D_\rho,\Omega^1(\mathfrak{g}))\times L_1^{2\alpha}(D_\rho,F)$ solves~\eqref{eq:EL-L-alpha-local} weakly in $D_\rho$. Then $(\tilde{A},\tilde{u})\in L_1^2(B_\rho,\Omega^1(\mathfrak{g}))\times L_1^{2\alpha}(B_\rho,F)$, and for all $(\tilde{B},\tilde{v})\in C_0^\infty\bigl(B_\rho,\Omega^1(\mathfrak{g})\bigr)\times C_0^\infty(B_\rho,\tilde{u}^*(TF))$, there holds
   \begin{equation}\label{eq:weak-Au-extend}
     \begin{dcases}
       \int_{B_\rho}\alpha(1+|\widetilde{\nabla}_{\tilde{A}}\tilde{u}|^2_{\tilde{h}})^{\alpha-1}\left\langle \widetilde{\nabla}_{\tilde{A}}\tilde{u},\widetilde{\nabla}_{\tilde{A}}\tilde{v} \right\rangle+\left\langle \tilde{\mu}(\tilde{u}),\widetilde{\nabla}\tilde{\mu}(\tilde{u})\tilde{v} \right\rangle=0,\\
       \int_{B_\rho}\alpha(1+|\widetilde{\nabla}_{\tilde{A}}\tilde{u}|^2_{\tilde{h}})^{\alpha-1}\left\langle \widetilde{\nabla}_{\tilde{A}}\tilde{u},\tilde{B}\tilde{\cdot}\tilde{u} \right\rangle+\left\langle F_{\tilde{A}},\widetilde{D}_{\tilde{A}}\tilde{B} \right\rangle=0.
     \end{dcases}
   \end{equation}
 \end{lem}
 \begin{proof}
   The weak form of~\eqref{eq:EL-L-alpha-local} is given by, for any $(B,v)\in L_1^2\bigcap C^0(D_{\rho},\Omega^1(\mathfrak{g}))\times L_1^{2\alpha}(D_{\rho},u^*(TF))$ and $v|_{\partial^0 D_{\rho}}\in T_uK$,
   \[
     0=\int_{D_{\rho}}\alpha(1+|\nabla_Au|^2)^{\alpha-1}\left\langle \nabla_Au,\nabla_Av+B\cdot u \right\rangle+\left\langle F_A,D_AB \right\rangle+\left\langle \mu(u),\nabla\mu(u)v \right\rangle.
   \]
   That is
   \begin{equation}\label{eq:weak-alpha}
     \begin{dcases}
       \int_{D_\rho}\alpha(1+|\nabla_Au|^2)^{\alpha-1}\left\langle \nabla_Au,\nabla_Av \right\rangle + \left\langle \mu(u),\nabla\mu(u)v \right\rangle=0,\\
       \int_{D_\rho} \alpha(1+|\nabla_Au|^2)^{\alpha-1}\left\langle \nabla_Au,B\cdot u \right\rangle+\left\langle F_A,D_AB \right\rangle=0.
     \end{dcases}
   \end{equation}

   Let us write down~\eqref{eq:weak-alpha} exactly in local Fermi coordinates. Note that the test functions $B$ and $v$ are vector valued, and we will test each component. For $v(x)=v^b(x)\partial_{f^b}(u(x))\in\Gamma(u^*TF)$, by \eqref{eq:nabla-A-u}, \eqref{eq:nabla-partial_i-v} and \eqref{eq:A-v},
   \[
     \left\langle \nabla_Au,\nabla_Av \right\rangle=h_{ac}u^a_{|i}\left( \partial_i v^c(x)+ v^b(x)\partial_i u^d(x)\Gamma_{bd}^c( u(x)) +v^b(x)A_{bi}^c(u(x))\right),
   \]
   where $u^a_{|i} \mathpunct{:}=\partial_iu^a+A_i^a(u)=\partial_iu^a+a_i^\beta \lambda_\beta^a(u)$. For any fix $b$, if we take $v^b(x)=\varphi(x)$, where $\varphi\in C^\infty(D_\rho)$ for $b\in\I_1$ and $\varphi\in C_0^\infty(D_\rho)$ for $b\in\I_2$, then by the first equation of~\eqref{eq:weak-alpha},
   \begin{equation}\label{eq:weak-u}
     0=\begin{multlined}[t]
       \int_{D_\rho}\alpha\left( 1+|\nabla_Au|^2 \right)^{\alpha-1}\big( h_{ab}(u)\partial_i\varphi u^a_{|i} + h_{ad}(u)\varphi u^a_{|i}\partial_iu^c\Gamma_{bc}^d(u) \\
       + h_{ad}(u)\varphi u^a_{|i} A_{bi}^d \big) + \varphi\mu(u)\cdot[\nabla_{\partial_{f^b}}\mu](u).
     \end{multlined}
   \end{equation}

   By the definition of induced connection,
   \[
     D_AB=dB+[A\wedge B]=(-\partial_jB_i+[A_i,B_j])dx^i\wedge dx^j.
   \]
   Recall also that
   \[
     F_A=dA+A\wedge A=(\partial_iA_j+A_iA_j)dx^i\wedge dx^j \mathpunct{:}=A_{j;i}dx^i\wedge dx^j,\quad
     F_{ij}=\frac{1}{2}\left( A_{j;i}-A_{i;j} \right),
   \]
   In order to write down the above equations locally, suppose $B_i(x)=b_i^\alpha(x) v_\alpha$, $\left\langle v_\alpha,v_\beta \right\rangle=\delta_{\alpha\beta}$, $\partial_iv_\alpha=0$ and $[v_\alpha,v_\beta]=g_{\alpha\beta}^\gamma v_\gamma$, then
   \begin{align*}
     2F_A&=2F_{ij}dx^i\wedge dx^j=(\partial_iA_j-\partial_jA_i+[A_i,A_j])dx^i\wedge dx^j\\
     %&=\left((\partial_ia_j^\alpha-\partial_ja_i^\alpha) v_\alpha+(a_j^\alpha\partial_i v_\alpha-a_i^\alpha\partial_jv_\alpha)+(a_i^\alpha a_j^\beta-a_j^\alpha a_i^\beta) g_{\alpha\beta}^\gamma v_\gamma\right)dx^i\wedge dx^j\\
      &=\left( (\partial_ia_j^\gamma-\partial_ja_i^\gamma)+2a_i^\alpha a_j^\beta g_{\alpha\beta}^\gamma \right)v_\gamma dx^i\wedge dx^j
      \mathpunct{:}=2F_{ij}^\gamma v_\gamma dx^i\wedge dx^j,\\
     D_AB&=(\partial_iB_j+[A_i,B_j])dx^i\wedge dx^j
     %&=\left( \partial_ib_j^\alpha v_\alpha+b_j^\alpha\partial_i v_\alpha+a_i^\alpha b_j^\beta g_{\alpha\beta}^\gamma v_\gamma \right) dx^i\wedge dx^j\\
     =\left( \partial_i b_j^\gamma+a_i^\alpha b_j^\beta g_{\alpha\beta}^\gamma \right)v_\gamma dx^i\wedge dx^j,\\
     \left\langle F_A,D_AB \right\rangle&=\left( \partial_i b_j^\gamma+a_i^\alpha b_j^\beta g_{\alpha\beta}^\gamma \right)F_{ij}^\gamma.
   \end{align*}
   Note also that $A_i^\sharp(u(x))=a_i^\alpha(x)V_\alpha(u(x))$ and $B_i\cdot u=B_i^\sharp(u)=b_i^\alpha V_\alpha(u)=b_i^\alpha(x)\lambda_\alpha^a(u)\partial_{f^a}(u)$, thus we can write $\left\langle \nabla_Au,B\cdot u \right\rangle$ locally as follows:
   \[
     \left\langle \nabla_Au,B\cdot u \right\rangle=h_{ab}(u)u^a_{|i}b_i^\beta(x)\lambda_\beta^b(u),
   \]
   Therefore, for any fixed $j,\beta$, if we take $b_j^\beta=\vartheta$, such that
   \[
     \vartheta\in
     \begin{dcases}
       C^\infty(D_\rho),  &j=1,\\
       C_0^\infty(D_\rho),&j=2.
     \end{dcases}
   \]
   then by the second equation of~\eqref{eq:weak-alpha},
   \begin{equation}\label{eq:weak-A}
     0=\int_{D_\rho}\alpha\left(1+|\nabla_Au|^2\right)^{\alpha-1}\vartheta h_{ab}(u) u^a_{|j}\lambda_\beta^b(u)+F_{ij}^\beta\partial_i\vartheta+F_{ij}^\gamma a_i^\delta g_{\beta\delta}^\gamma\vartheta.
   \end{equation}

   The next step is to show that, if we extend $h_{ab}$, $\Gamma_{ab}^c, A, u, \mu$ as before, and use prime to distinguish the equations obtained by replacing all quantities in \eqref{eq:weak-u} and \eqref{eq:weak-A} with their extensions, then $\tilde{u}$ and $\tilde{A}$ satisfy \eqref{eq:weak-u}$'$ and \eqref{eq:weak-A}$'$  respectively. Clearly,~\eqref{eq:weak-u}$'$ holds for $(\tilde{A},\tilde{u})$ and any  $\varphi\in C_0^\infty(B_\rho)$ when $b\in\I_1$ and any $\varphi\in C_0^\infty(D_\rho)$ with $\varphi\equiv0$ on $\partial^0D_\rho$ when $b\in\I_2$. We only need to check that when $b\in\I_2$,~\eqref{eq:weak-u}$'$ holds for any $\varphi\in C_0^\infty(B_\rho)$. Write
   \[
     \varphi=\varphi_e+\varphi_o,\quad
     \varphi_e(x)\mathpunct{:}=\frac{1}{2}\left( \varphi(x)+\varphi(x^*) \right),\quad
     \varphi_o(x)\mathpunct{:}=\frac{1}{2}\left( \varphi(x)-\varphi(x^*) \right),
   \]
   then clearly, for $x\in D_\rho^-$,
   \[
     \partial_i\varphi_o(x)=(-1)^i\partial_i\varphi_o(x^*),\quad \partial_i\varphi_e(x)=(-1)^{i-1}\partial_i\varphi_e(x^*).
   \]
   Note that for $b\in\I_2$, it is easy to check the parity of the components in~\eqref{eq:weak-u}$'$, for any $x\in D_\rho^-$,
   \begin{align*}
     \tilde{h}_{ab}(\tilde{u})\tilde{u}^a_{|i}|_{x}&=(-1)^ih_{ab}(u)u^a_{|i}|_{x^*}\\
     \tilde{h}_{ad}(\tilde{u})\tilde{u}^a_{|i}\partial_i\tilde{u}^c\tilde{\Gamma}_{bc}^d(\tilde{u})|_x& =-h_{ad}(u)u^a_{|i}\partial_i u^c \Gamma_{bc}^d(u)|_{x^*}\\
     \tilde{h}_{ad}(\tilde{u})\tilde{u}^a_{|i}\tilde{A}_{bi}^d(\tilde{u})|_x&=-h_{ad}(u)u^a_{|i} A_{bi}^d(u)|_{x^*}.
   \end{align*}
   Now, clearly, $|\widetilde{\nabla}_{\tilde{A}}\tilde{u}|^2_{\tilde{h}}(x)=|\nabla_Au|^2_h(x^*)$, for any $x^*\in D_\rho$. We can compute the extended weak equation ~\eqref{eq:weak-u}$'$ as
   \begin{align*}
    &\begin{multlined}
      \int_{B_\rho}\alpha(1+|\widetilde{\nabla}_{\tilde{A}}\tilde{u}|^2_{\tilde{h}})^{\alpha-1}\big( \tilde{h}_{ab}(\tilde{u})\partial_i\varphi\tilde{u}^a_{|i}+\tilde{h}_{ad}(\tilde{u})\varphi\tilde{u}^a_{|i}\partial_i\tilde{u}^c\tilde{\Gamma}_{cb}^d(\tilde{u})\\
      +\tilde{h}_{ad}(\tilde{u})\varphi \tilde{u}^a_{|i}\tilde{A}_{bi}^d\big)+\varphi\tilde{\mu}(\tilde{u})\cdot[\widetilde{\nabla}_{\partial _{f^b}}\tilde{\mu}](\tilde{u})
    \end{multlined}\\
    &=\begin{multlined}[t]
      \int_{B_\rho}\alpha(1+|\widetilde{\nabla}_{\tilde{A}}\tilde{u}|^2_{\tilde{h}})^{\alpha-1}\big( \tilde{h}_{ab}(\tilde{u})\partial_i(\varphi_e+\varphi_o)\tilde{u}^a_{|i}+\tilde{h}_{ad}(\tilde{u})(\varphi_e+\varphi_o)\tilde{u}^a_{|i}\partial_i\tilde{u}^c\tilde{\Gamma}_{cb}^d(\tilde{u})\\
      +\tilde{h}_{ad}(\tilde{u})(\varphi_e+\varphi_o) \tilde{u}^a_{|i}\tilde{A}_{bi}^d\big)+(\varphi_e+\varphi_o)\tilde{\mu}(\tilde{u})\cdot[\widetilde{\nabla}_{\partial_{f^b}}\tilde{\mu}](\tilde{u})
    \end{multlined}\\
    &=\begin{multlined}[t]
      \int_{B_\rho}\alpha(1+|\widetilde{\nabla}_{\tilde{A}}\tilde{u}|^2_{\tilde{h}})^{\alpha-1}\big( \tilde{h}_{ab}(\tilde{u})\partial_i\varphi_o\tilde{u}^a_{|i}+\tilde{h}_{ad}(\tilde{u})\varphi_o\tilde{u}^a_{|i}\partial_i\tilde{u}^c\tilde{\Gamma}_{cb}^d(\tilde{u})\\
      +\tilde{h}_{ad}(\tilde{u})\varphi_o \tilde{u}^a_{|i}\tilde{A}_{bi}^d\big)+\varphi_o\tilde{\mu}(\tilde{u})\cdot[\widetilde{\nabla}_{\partial_{f^b}}\tilde{\mu}](\tilde{u})
    \end{multlined}\\
    &=\begin{multlined}[t]
      2\int_{D_\rho}\alpha(1+|\nabla_{A} u|^2_{h})^{\alpha-1}\big(h_{ab}(u)\partial_i\varphi_ou^a_{|i}+ h_{ad}(u)\varphi_ou^a_{|i}\partial_i u^c\Gamma_{cb}^d(u)\\
      + h_{ad}(u)\varphi_o u^a_{|i} A_{bi}^d\big)+\varphi_o\mu(u)\cdot[\nabla_{\partial_{f^b}}\mu](u)
    \end{multlined}\\
    &=0,
   \end{align*}
   the last equality follows from~\eqref{eq:weak-u} and the fact that $\varphi_o=0$ on $\partial^0D_\rho$. This shows that the extended solution $(\tilde{A},\tilde{u})$ solves~\eqref{eq:weak-u}$'$ weakly.

   Lastly, we verify~\eqref{eq:weak-A}$'$ for extended $(\tilde{A},\tilde{u})$. Actually, we only need the following symmetries for $j=2$. By the symmetry of $a_i^\alpha$ (see \eqref{eq:a-i-alpha}),
   \[
     \tilde{F}_{ij}^\gamma(x)=\frac{1}{2}\left( \partial_{x^i}\tilde{a}_j^\gamma(x)-\partial_{x^j}\tilde{a}_i^\gamma(x)+2\tilde{a}_i^\alpha(x)\tilde{a}_j^\beta(x)g_{\alpha\beta}^\gamma  \right)
     =(-1)^{i+j}F_{ij}^\gamma(x^*).
   \]
   It is easy to check, the following symmetries
   \begin{align*}
     \tilde{F}_{ij}^\gamma(x)\tilde{a}_i^\delta (x)&=(-1)^{j-1}F_{ij}^\gamma(x^*)a_i^\delta(x^*),\\
     \tilde{h}_{ab}(\tilde{u})\tilde{u}^a_{|j}\tilde{\lambda}_\beta^b(\tilde{u})&=(-1)^{j-1} h_{ab}(u(x^*))u^a_{|j}(x^*) \lambda_\beta^b(u(x^*)).
   \end{align*}
   With these parities in hand, we can decompose $\vartheta$ into even part and odd part as $\varphi$ and the verification of~\eqref{eq:weak-A}$'$ is the same.
 \end{proof}
 Now, we are ready to prove \autoref{thm:alpha-smoothness}. We first rewrite the extended weak equation to standard form, and then check the condition~\eqref{eq:coupled-condi} is satisfied. \autoref{thm:Morrey} shows that the weak solution is strong, and we can bootstrap the regularity of the strong solution to show the smoothness up to the boundary.
 \begin{proof}[Proof of~\autoref{thm:alpha-smoothness}]
   By \autoref{lem:ps}, for $\alpha>1$, there exists a weak solution $(A_\alpha,\phi_\alpha)\in \mathscr{A}_1^2\times \mathscr{S}_{1,K}^{2\alpha}$ of \eqref{eq:EL-L-alpha-global}. We will improve the regularity of this weak solution module gauge and prove \autoref{thm:alpha-smoothness}.
   \step We first show the $L_2^2$-interior regularity, which is a direct application of \autoref{thm:Morrey}. Locally, if we write the solution as $(A,u)$, where $A$ is a $\mathfrak{g}$-valued 1-form on $U=B_\rho$ and $u$ is a map from $U$ to $F$, then the equation of $(A,u)$ is given by \eqref{eq:weak-u} and \eqref{eq:weak-A}. That is, for $(\vartheta,\varphi)\in C_0^\infty(B_\rho)$ (note that $C_0^\infty(B_\rho)$ is dense in $L_1^2\cap C^0(B_\rho)\times L_1^{2\alpha}(B_\rho)$),
   \begin{equation}\label{eq:weak-A-u}
     \begin{dcases}
       0=\int_{B_\rho} \left\{ \partial_i\vartheta_j^\beta F_{ij}^\beta+\vartheta_j^\beta\left( \Upsilon h_{ab}(u)u^a_{|j}\lambda_\beta^b(u)+F_{ij}^\gamma a_i^\delta g_{\beta\delta}^\gamma \right) \right\}\\
       0=\begin{multlined}[t]\int_{B_\rho}\partial_i\varphi^b\cdot\Upsilon h_{ab}(u) u^a_{|i}\\
         +\int_{B_\rho}\varphi^b\Bigl( \Upsilon h_{ad}(u) u^a_{|i}\left( \partial_iu^c\Gamma_{cb}^d(u)+A_{bi}^d(u) \right)+\mu(u)\cdot[\nabla_{\partial_{f^b}}\mu](u)\Bigr),
       \end{multlined}
     \end{dcases}
   \end{equation}
   where $\Upsilon=\alpha(1+\lvert \nabla_Au \rvert^2)^{\alpha-1}$. It is well-known in Yang--Mills theory that $D_A^*F_A=0$ is not a strict elliptic equation of $A$, we need to module the gauge action. Applying \autoref{thm:coulomb} to $A$ on $B_\rho$, we can assume further that $A$ is in Coulomb gauge. Since under local Coulomb gauge, there holds, for all $\beta=1,2,\ldots,m$, $\sum_{i=1}^2\partial_ia_i^\beta(x)=0$. Integration by parts shows
   \begin{align*}
     2\int_{B_\rho}\partial_i\vartheta_j^\beta F_{ij}^\beta
     &=\int_{B_\rho}\partial_i\vartheta_j^\beta\left( \partial_ia_j^\beta-\partial_ja_i^\beta+2a_i^\gamma a_j^\delta g_{\gamma\delta}^\beta \right)\\
     &=\int_{B_\rho}\partial_i\vartheta_j^\beta \partial_ia_j^\beta-2\vartheta_j^\beta a_i^\gamma \partial_ia_j^\delta g_{\gamma\delta}^\beta .
   \end{align*}
   Thus, under Coulomb gauge, \eqref{eq:weak-A-u} transforms to
   \begin{equation}\label{eq:weak-A-u-Coulomb}
     \begin{dcases}
       0=\int_{B_\rho} \left\{ \partial_i\vartheta_j^\beta \partial_ia_j^\beta+2\vartheta_j^\beta\left( -a_i^\gamma\partial_ia_j^\delta g_{\gamma\delta}^\beta+\Upsilon h_{ab}(u)u^a_{|j}\lambda_\beta^b(u)+F_{ij}^\gamma a_i^\delta g_{\beta\delta}^\gamma \right) \right\}\\
       0=\begin{multlined}[t]\int_{B_\rho}\partial_i\varphi^b\cdot\Upsilon h_{ab}(u) u^a_{|i}\\
         +\int_{B_\rho}\varphi^b\Bigl( \Upsilon h_{ad}(u) u^a_{|i}\left( \partial_iu^c\Gamma_{cb}^d(u)+A_{bi}^d(u) \right)+\mu(u)\cdot[\nabla_{\partial_{f^b}}\mu](u)\Bigr),
       \end{multlined}
     \end{dcases}
   \end{equation}
   where $\Upsilon=\alpha\left(1+h_{ab}(u)(\partial_iu^a+a_i^\beta\lambda_\beta^a(u))(\partial_iu^b+a_i^\gamma\lambda_\gamma^b(u))\right)^{\alpha-1}$, $u_{|i}^a=\partial_iu^a+a_i^\beta \lambda_\beta^a(u)$, $F_{ij}^\gamma=(\partial_ia_j^\gamma-\partial_ja_i^\gamma)+2a_i^\beta a_j^\delta g_{\beta\delta}^\gamma$, $A_{bi}^d(u)=a_i^\beta\left(\partial_{f^b}\lambda_\beta^d(u)+\lambda_\beta^c(u)\Gamma_{bc}^d(u)\right)$, and $h$, $\Gamma$, $\lambda$, $\mu$ are smooth functions of $u$. By definition $[v_\beta,v_\gamma]=g_{\beta\gamma}^\delta v_\delta$, $\left\{ g_{\beta\gamma}^\delta \right\}$ are called the structure constants of the Lie algebra.

   To apply~\autoref{thm:Morrey}, let $\Omega=B_\rho$ and
   \begin{align*}
     k_1&=2,\quad k_2=2\alpha,\\
     z &=(z_1,z_2)=(A,u),\quad z_{1,i}^\beta=a_i^\beta,\quad z_2^b=u^b\\
     \xi&=(\xi_1,\xi_2)=(\vartheta,\varphi),\quad\xi_{1,j}^\beta=\vartheta_j^\beta,\quad\xi_2^b=\varphi^b\\
     p &=(p_1,p_2)=(\nabla A,\nabla u),\quad p_{1,i;j}^\beta=\partial_ja_i^\beta,\quad p_{2;j}^b=\partial_ju^b,
   \end{align*}
   then the equation~\eqref{eq:weak-A-u-Coulomb} is of form~\eqref{eq:coupled-weak} with coefficients
   \begin{align*}
     q(x,z,p)  &=(q_1(x,z,p), q_2(x,z,p))\\
     q_1(x,z,p)&=p_1\\
     q_2(x,z,p)&=\alpha(1+h(z_2)(p_2+z_1\lambda(z_2))^2)^{\alpha-1}h(z_2)(p_2+z_1\lambda(z_2))\\
     w(x,z,p)  &=(w_1(x,z,p),w_2(x,z,p))\\
     w_1(x,z,p)&=\alpha(1+h(z_2)(p_2+z_1\lambda(z_2))^2)^{\alpha-1}h(z_2)\lambda(z_2)(p_2+z_1\lambda(z_2))+(p_1+z_1^2)z_1\\
     w_2(x,z,p)&=\alpha(1+h(z_2)(p_2+z_1\lambda(z_2))^2)^{\alpha-1}h(z_2)(p_2+z_1\lambda(z_2))\\
               &\qquad\cdot\left(p_2\Gamma(z_2)+z_1\left( \nabla\lambda(z_2)+\lambda(z_2)\Gamma(z_2) \right)\right)+\mu(z_2)\cdot\nabla\mu(z_2),
   \end{align*}
   where $h$, $\Gamma$, $\lambda$ and $\mu$ are smooth functions of $z_2$ and $h$ is positive definite, which is bounded from above and below. The verification of condition~\eqref{eq:coupled-condi} is tedious but straightforward. We illustrate by the computation of $w_z$. A direct computation shows that for some $\Lambda(R)$ depending on $\alpha$, the geometry of $F$, $G$ and $\mu$,
   \begin{alignat*}{2}
     |\partial_{z_1}w_1| &\leq \Lambda(R)\left(V_1^{k_1-1}+V_2^{k_2-2}\right),
                         &\quad
     |\partial_{z_2}w_1| &\leq \Lambda(R)V_2^{k_2-1},\\
     |\partial_{z_1}w_2| &\leq \Lambda(R)V_2^{k_2-1},
                         &\quad
     |\partial_{z_2}w_2| &\leq \Lambda(R)V_2^{k_2},
   \end{alignat*}
   where $R$ is the upper bound for $x$ and $z$, i.e., $\lvert x \rvert^2+\lvert z \rvert^2\leq R^2$.  It is clear that, for any vector $\pi=(\pi_1,\pi_2)$, we have
   \begin{align*}
     \pi\cdot\frac{\partial(w_1,w_2)}{\partial(z_1,z_2)}\cdot\pi^T&=\sum_{a,b}\pi_a\cdot\partial_{z_a}w_b\cdot\pi_b^T\\
                                                                  &\leq \Lambda(R)\left( \left(V_1^{k_1-1}+V_2^{k_2-2}\right) \lvert \pi_1 \rvert^2+2V_2^{k_2-1}\lvert \pi_1 \rvert \lvert \pi_2 \rvert+V_2^{k_2}\lvert \pi_2 \rvert^2 \right)\\
                                                                  &\leq 2\Lambda(R)\left( \left(V_1^{k_1-1}+V_2^{k_2-2}\right) \lvert \pi_1 \rvert^2+ V_2^{k_2}\lvert \pi_2 \rvert^2 \right)\\
                                                                  &\leq 2\Lambda(R)\left( \left( V_1^{k_1}+V_2^{k_2-1} \right)\lvert \pi_1 \rvert^2+V_2^{k_2}\lvert \pi_2 \rvert^2 \right).
   \end{align*}
   The verification of other conditions is more or less the same. Moreover, the additional regularity assumption in~\autoref{thm:Morrey} can be shown as follows: for $A\in L_1^2(B_\rho,\Omega^1(\mathfrak{g}))$ and $u\in L_2^{2\alpha}(B_\rho, u^*(TF))$ solve \eqref{eq:weak-Au-extend} weakly, then $A$ solves
   \[
     \Delta A-\left\langle dA,A \right\rangle-\left\langle A,[A,A] \right\rangle+\alpha\left( 1+\lvert \nabla_Au \rvert^2 \right)^{\alpha-1}\left\langle \nabla_Au,u \right\rangle=0,
   \]
   weakly, and note that
   \begin{gather*}
     \left\langle dA,A \right\rangle\in L^p,\quad1<p<2,\qquad\left\langle A,[A,A] \right\rangle\in L^q,\quad1<q<+\infty,\\
     \nabla_Au=\nabla u+A\cdot u\in L^{2\alpha},\quad\left( 1+\lvert \nabla_Au \rvert^2 \right)^{\alpha-1}\in L^{\frac{\alpha}{\alpha-1}},\\
     \alpha\left( 1+\lvert \nabla_Au \rvert^2 \right)^{\alpha-1}\left\langle \nabla_Au,u \right\rangle\in L^{\frac{2\alpha}{2\alpha-1}}.
   \end{gather*}
   Thus $\Delta A\in L^{\frac{2\alpha}{2\alpha-1}}$ and $A\in L_2^{\frac{2\alpha}{2\alpha-1}}\hookrightarrow L_1^{\frac{2\alpha}{\alpha-1}}\hookrightarrow L_1^{4\alpha}$ when $\alpha-1>0$ is small enough. Note that, by Sobolev embedding $(A,u)\in C^\mu(B_\rho)$ for some $\mu\in(0,1)$. Finally, we apply~\autoref{thm:Morrey} to conclude that $(A,u)\in L_2^2(B_\rho)$, this shows the $L_2^2$-interior regularity of the weak solution of $\alpha$-Yang--Mills--Higgs fields.

   \step Next, we show the $L_2^2$-boundary regularity. At a boundary point $x_0\in\partial\Sigma$, since $\partial\Sigma$ is smooth, we can assume that the coordinate chart at $x_0$ is the upper half disc $D_\rho$ centered at origin (since we can ``flatten out'' a piece of the boundary by a bi-Lipschitz map). By \autoref{lem:bdry-reg}, the equation of extended solution $(\tilde{A}, \tilde{u})$ is given by \eqref{eq:weak-Au-extend}. Comparing to \eqref{eq:weak-alpha}, we know that $(\tilde{A},\tilde{u})\in L_1^2(B_\rho,\Omega^1(\mathfrak{g}))\times L_1^{2\alpha}(B_\rho,F)$ satisfies a system of equations similar to \eqref{eq:weak-A-u-Coulomb}, with the coefficients $h$, $\Gamma$, $\lambda$ and $\mu$ extends properly as in \autoref{sec:smoothness}. Although these extended coefficients are $C^\infty$-smooth on $V_\delta\setminus K$, where $V_\delta\subset F$ is a tubular neighborhood of $K$, they are not $C^\infty$-smooth on $V_\delta$ in general. However, under the assumption $K\subset F$ is a totally geodesic sub-manifold, we have $\tilde{h}$, $\tilde{\lambda}$ and $\tilde{\mu}$ are in $C^{1,\alpha'}(V_\delta)$ and $\tilde{\Gamma}\in C^{\alpha'}(V_\delta)$ for some $\alpha'\in(0,1)$. In particular, the regularity requirement of \autoref{thm:Morrey} is satisfied. Moreover, $\tilde{h}$, $\tilde{\lambda}$, $\tilde{\mu}$ and $\tilde{\Gamma}$ as functions of $\tilde{u}$ may be multi-valued, but since $u$ is continuous, if we take $\rho$ small enough, then they are still single-valued as functions of $\tilde{u}$ when restricted to $\tilde{u}(B_\rho)$. With these properties of the extended equations in mind, we can apply the \autoref{thm:Morrey} to show that $(\tilde{A},\tilde{u})\in L_2^2(B_{\rho})$ and $(A,u)\in L_2^2(D_\rho)$.

   \step As long as we show the $L_2^2$-regularity,~\eqref{eq:A-u-strong} holds strongly. If $\alpha-1$ is small, then the linear operator
   \begin{align*}
     \Delta_{(A,u)}\mathpunct{:}L_{k+2}^p(U,F)&\to L_k^p(U,F)\\
     v&\mapsto \Delta_{\Sigma}v-2(\alpha-1)\frac{\left\langle \nabla_A^2v, \nabla_Au \right\rangle\nabla_Au }{1+\lvert \nabla_Au \rvert^2}
   \end{align*}
   is invertible. Now the smoothness of weak solution can be proved by standard bootstrap argument with up to the boundary estimates. In fact, $du\in L_1^2(U,F)$, $A\in L_1^2(U,\mathfrak{g})$, $u\in C^0(U,F)$ and since $\mu$ is smooth, $\mu(u)\in L_1^2\cap C^0(U)$. By the Sobolev multiplications $L_1^2\otimes L_1^2\to L_1^p$ for some $p$ slightly smaller than $2$, $\Phi_\alpha(A,u)\in L_1^p(U,F)$. The inevitability of $\Delta_{(A,u)}$ shows that $u\in L_3^p(U,F)$. Also, since $A\in L_2^2\bigl(U,\Omega^1(\mathfrak{g})\bigr)$, it is easy to show $\Psi_\alpha(A,u)\in L^{3p}(U,\mathfrak{g})$. Thus,~\eqref{eq:A-u-strong} implies that $A\in L_2^{3p}(U,\mathfrak{g})$. Now, $\Phi_\alpha(A,u)\in L_2^p(U,F)$ and so $u\in L_4^p(U,F)$ this time (we need~\autoref{lem:Lp-Dirichlet-Neumann} if $U\cap\partial\Sigma\neq\emptyset$), which in turn gives $\Psi_\alpha(A,u)\in L_2^{2p}(U,\mathfrak{g})$ and $A\in L_4^{2p}(U,\mathfrak{g})$. Iterating like this again and again, we can show that $(A,u)$ is smooth in $U$ \emph{up to the boundary}.

   \step We should note that the above smoothness requires that $A$ is under some $L_2^2$-Coulomb gauge. Let $\left\{ U_\beta \right\}$ be an open cover of $\Sigma'$, where each $U_\beta$ is an open ball $B_\rho$ such that the above interior smooth regularity holds under some $L_2^2$-Coulomb gauge. Since the $\alpha$-YMH functional is invariant under gauge transformation, we can patch these local gauges together to obtain a global gauge $\tilde{S}\in \mathscr{G}_2^2(\Sigma')$ in the same way as~\cite{Song2011CriticalYangMillsHiggs}*{Sect.~3}, such that $(\tilde{S}^*A_\alpha,\tilde{S}^*\phi_\alpha)$ is smooth on $\Sigma'$. In the same manner, we can patch the local gauges over an open cover of $\Sigma$ to obtain a global gauge $\tilde{S}\in \mathscr{G}_2^2(\Sigma)$, such that $(\tilde{S}^*A_\alpha,\tilde{S}^*\phi_\alpha)$ is smooth on $\Sigma$ up to the boundary. This finishes the proof of~\autoref{thm:alpha-smoothness}.
 \end{proof}
 \section{The main estimates}\label{sec:eps-regularity}
 In this section we give some local uniform (independent of $\alpha$) estimates for critical points of $\mathcal{L}_\alpha$, which server as a preparation of blow-up analysis. We focus on the local boundary estimates, because the corresponding interior one follows as in~\cite{Song2011CriticalYangMillsHiggs}*{Sect.~4}. Suppose $U$ is a domain in $\Sigma$ and under a fixed trivialization we write $\phi(x)=(x,u(x))$ and $\nabla_A=d+A$ as before. Since $u\in L_1^{2\alpha}(U,F)\subset C^0(U,F)$, for $x_0=0\in\partial\Sigma$, we can take Fermi coordinates $(f^1,\ldots,f^n)$ on an open neighborhood $V$ of $p=u(x_0)\in K$ such that $V\cap K=\left\{ f^{k+1}=\cdots=f^n=0 \right\}$ as in~\autoref{sec:smoothness}.

 Take polar coordinates $(r,\theta)$ on $U$, we always assume $A$ is in Coulomb gauge with estimate~\eqref{item:coulomb-est} in~\autoref{thm:coulomb} holds. Under these assumptions, the Euler--Lagrange equation of $\mathcal{L}_\alpha$ is given by (see~\eqref{eq:A-u-strong}),
 \begin{equation}\label{eq:A-u-strong-Fermi}
   \begin{cases}
     \Delta_\Sigma u-2(\alpha-1)\frac{\left\langle \nabla^2_Au,\nabla_Au \right\rangle\nabla_Au}{1+|\nabla_Au|^2}-\Phi_\alpha(A,u)=0,&x\in U\\
     \Delta A-\Psi_\alpha(A,u)=0,&x\in U\\
     A_2=0, &x\in\partial U\\
     \frac{\partial A_1}{\partial\nu}=0,&x\in\partial^0 U\\
     \frac{\partial u^a}{\partial\nu}=0,\quad a=1,2,\ldots,k,&x\in\partial^0 U\\
     u^a=0,\quad a=k+1,\ldots,n,&x\in\partial^0 U
   \end{cases}
 \end{equation}
 where $\Phi_\alpha$ and $\Psi_\alpha$ are defined by~\eqref{eq:Phi} and~\eqref{eq:Psi} respectively.

 Similarly, the local equation for critical points of $\mathcal{L}$ is given in the following lemma.
 \begin{lem}
   Suppose $(A,\phi)$ is a critical point of $\mathcal{L}$ on $\mathscr{A}\times \mathscr{S}_K$, then locally, when we choose Coulomb gauge in $U$, that is
   \[
     \begin{cases}
       d^*A=0,&x\in U\\
       \nu\lh A=0,&x\in\partial U,
     \end{cases}
   \]
   the Euler--Lagrange equation can be written as:
   \begin{equation}\label{eq:loc}
     \begin{cases}
       \Delta_\Sigma u-\Phi_1(A,u)=0,&x\in U\\
       \Delta A-\Psi_1(A,u)=0,&x\in U\\
       A_2=0,&x\in\partial U\\
       \frac{\partial A_1}{\partial\nu}=0,&x\in\partial^0 U\\
       \frac{\partial u^a}{\partial\nu}\perp T_uK,\quad a=1,\ldots,k&x\in\partial^0 U\\
       u^a=0,\quad a=k+1,\ldots,n&x\in\partial^0U,
     \end{cases}
   \end{equation}
   where
   \[
     \Phi_1(A,u)=\Gamma(u)(du,du)+2A\cdot du+A\cdot A\cdot u+\mu(u)\cdot\nabla\mu(u)
   \]
   and
   \[
     \Psi_1(A,u)=\left\langle dA,A \right\rangle+\left\langle A,[A,A] \right\rangle-\left\langle \nabla_Au,u \right\rangle.
   \]
 \end{lem}
 \subsection{\texorpdfstring{$\epsilon$}{epsilon}-regularity estimates}
 The main estimates in Sacks-Uhlenbeck's method is the so-called $\epsilon$-regularity theorem. Here we prove an analogy for $\phi$ with small energy $\lVert \nabla_A\phi \rVert_{L^2(U)}$.

 \begin{lem}[$\epsilon$-regularity]\label{lem:eps-regulairty}
   There exist $\epsilon_0>0$ and $\alpha_0>1$ such that if $(A,u)\in \mathscr{A}(U)\times \mathscr{S}_K(U):=\mathscr{A}|_{U}\times \mathscr{S}_K|_U$ is a smooth pair satisfying~\eqref{eq:A-u-strong-Fermi} for $1\leq\alpha<\alpha_0$ with $\lVert \nabla_A u \rVert_{L^2(U,\Omega^1(F))}<\epsilon_0$ and $\mathcal{L}_\alpha(A,u;U)\leq\Lambda<+\infty$, then for any $U'\subset\!\subset U$ and $p>1$, the following estimate holds uniformly in $1\leq\alpha<\alpha_0$,
   \[
     \lVert u - \bar u \rVert_{L_2^p(U', F)}\leq C\left( \lVert \nabla_Au \rVert_{L^p(U,\Omega^1(F))}+\lVert F_A \rVert_{L^p(U,\Omega^2(\mathfrak{g}))}+1 \right),
   \]
   where $\bar u$ is the integral mean over $U$ and $C>0$ is a constant depending on $U$, $U'$, $F$, $\Lambda$, $\lVert \mu \rVert_{L_1^\infty(F)}$, $p$, $\alpha_0$, $\epsilon_0$.
 \end{lem}
 \begin{rmk}
   Note that~\eqref{eq:A-u-strong-Fermi} and \eqref{eq:loc} require that $A$ is in Coulomb gauge. We remark that when the radius of $U$ is small enough, this is always satisfied.

   In fact, $\mathcal{L}_\alpha(A,u;U)\leq \Lambda<+\infty$, in particular, $\lVert F_{A} \rVert_{L^2(U)}^2\leq \Lambda$. Thus, there exists a small constant $r_0$ (depending only on $\Lambda$, the geometry of $\Sigma$ and $\delta_0$), such that $\lVert F_A \rVert_{L^1(U)}\leq \delta_0$, provided that the radius of $U$ is smaller than $r_0$, so we may assume $A$ is in Coulomb gauge by~\autoref{thm:coulomb}.
 \end{rmk}
 \begin{proof}
   Since the interior case can be proved by a minor modification of the following boundary case, we assume $U$ is a upper half disc centered at $x_0=0\in\partial\Sigma$. As $F$ is compact and embedded into Euclidean space, we can assume $\bar u=0$ without loss of generality. In particular, we have the following Poincar\'e inequality,
   \[
     \lVert u \rVert_{L^p(U)}\leq C(U,p)\lVert du \rVert_{L^p(U)}.
   \]
   Since $\nabla_Au=du+Au$,
   \begin{equation}\label{eq:u-nabla_Au-control}
     \lVert u \rVert_{L_1^p(U)}\leq C(U,F,p)\left( \lVert \nabla_Au \rVert_{L^p(U)}+\lVert A \rVert_{L^p(U)} \right).
   \end{equation}

   Suppose $\eta$ is a cutoff function supported on $U$, $\nu\cdot\nabla\eta=0$ on $\partial\Sigma\cap U$ and $\eta|_{U'}\equiv1$. Multiplying the equation of $u$ in~\eqref{eq:A-u-strong-Fermi} by $\eta$, a direct computation shows
   \begin{align*}
     \Delta_\Sigma(\eta u)&\leq C(U,U')\left( \eta\lvert \Delta_\Sigma u \rvert+\lvert du \rvert+\lvert u \rvert \right)\\
                          &\leq C(U,U')\left( \eta(\alpha-1)\lvert \nabla_A^2u \rvert+\eta\lvert \Phi_\alpha(A,u) \rvert+\lvert du \rvert+\lvert u \rvert \right)\\
                          &\leq C(U,U')\bigl( (\alpha-1) \lvert d^2(\eta u) \rvert +\lvert d(\eta u) \rvert\lvert du \rvert\\
                          &\myquad[7]+\lvert dAu \rvert +\lvert A du \rvert+\lvert A^2 u \rvert+\lvert \nabla\mu(u) \rvert\lvert \mu(u) \rvert+\lvert du \rvert+\lvert u \rvert\bigr)\\
                          &\leq C(U,U',F)\Bigl( (\alpha-1) \lvert d^2(\eta u) \rvert+\lvert d(\eta u) \rvert\lvert \nabla_Au \rvert+\lvert \nabla_Au \rvert\lvert A \rvert\\
                          &\myquad[17]+\lvert A \rvert+\lvert A^2u \rvert+\lvert dA \rvert+\lVert \mu \rVert_{L_1^\infty}+\lvert du \rvert+\lvert u \rvert\Bigr).
   \end{align*}
   Now, note that the boundary condition of $\eta u$ is either of homogeneous Dirichlet or Neumann type
   \[
     \begin{cases}
       \frac{\partial (\eta u^a)}{\partial\nu}=0,\quad a=1,2,\ldots,k,\\
       \eta u^b=0,\quad b=k+1,\ldots,n.
     \end{cases}
   \]
   By the standard $L^p$ estimate (see~\autoref{lem:Lp-Dirichlet-Neumann}), the Sobolev embedding $L_1^{p} \hookrightarrow L^{2p}$ and~\eqref{eq:u-nabla_Au-control},
   \begin{equation}\label{eq:eta-u-l2p-est}
     \begin{split}
       \lVert \eta u \rVert_{L_2^p(U)}&\leq C\bigl(U,U',F,\lVert \mu \rVert_{L_1^\infty},p\bigr)
       \begin{multlined}[t]
         \Bigl( (\alpha-1)\lVert d^2(\eta u) \rVert_{L^p(U)}+\lVert \lvert d(\eta u) \rvert\lvert \nabla_Au \rvert \rVert_{L^p(U)}\\
           +\lVert \lvert A \rvert\lvert \nabla_Au \rvert \rVert_{L^p(U)}+\lVert A \rVert_{L_1^{p}(U)}\\
         +\lVert A^2u \rVert_{L^p(U)}+\lVert \nabla_Au \rVert_{L^p(U)}+1 \Bigr).
       \end{multlined}
     \end{split}
   \end{equation}
   First, for $1<p<2$, by the Sobolev embedding $L_1^p \hookrightarrow L^{p^*}$, $p^*=2p/(2-p)$ and H\"older's inequality,
   \begin{align*}
     \lVert \lvert d(\eta u) \rvert\lvert \nabla_Au \rvert \rVert_{L^p(U)}
    & \leq\lVert d(\eta u) \rVert_{L^{p^*}(U)}\lVert \nabla_Au \rVert_{L^2(U)}\\
    & \leq C(U,p)\lVert d(\eta u) \rVert_{L_1^p(U)}\lVert \nabla_Au \rVert_{L^2(U)},\\
    \lVert \lvert A \rvert\lvert \nabla_Au \rvert \rVert_{L^p(U)}
    & \leq C(U,p)\lVert A \rVert_{L_1^p(U)}\lVert \nabla_Au \rVert_{L^2(U)},
    \intertext{and}
    \lVert A^2u \rVert_{L^p(U)}
    & \leq\lVert \lvert A \rvert\cdot\lvert Au \rvert \rVert_{L^p(U)}\leq C(F)\lVert A \rVert_{L_1^p(U)}\lVert A \rVert_{L_1^2(U)}.
   \end{align*}
   Since $A$ is in Coulomb gauge in $U$, by~\eqref{item:coulomb-est} of~\autoref{thm:coulomb}
   \[
     \lVert A \rVert_{L_1^p(U)}\leq C\lVert F_A \rVert_{L^p(U)},\quad
     \lVert A \rVert_{L_1^2(U)}\leq C\lVert F_A \rVert_{L^2(U)}.
   \]
   Plugging these estimates into~\eqref{eq:eta-u-l2p-est}, when $\alpha_0-1$ is small enough,
   \begin{align*}
     \lVert \eta u \rVert_{L_2^p(U)}&\leq C
     \begin{multlined}[t]
       \Bigl[ \lVert \nabla_Au \rVert_{L^2(U)}\left( \lVert d(\eta u) \rVert_{L_1^p(U)}+\lVert A \rVert_{L_1^p(U)}\right)\\
       +\lVert F_A \rVert_{L^p(U)}+\lVert \nabla_Au \rVert_{L^p(U)}+1 \Bigr],
     \end{multlined}
   \end{align*}
   where $C>0$ is a constant depending on $U,U',F,\Lambda,\lVert \mu \rVert_{L_1^\infty},p$ and $\alpha_0$. Therefore, if we take $\lVert \nabla_Au \rVert_{L^2(U)}\leq \epsilon_0$ small enough (in particular, it depends on $\alpha_0$), then we can employ the estimate of Coulomb gauge again to conclude
   \[
     \lVert u \rVert_{L_2^p(U')}\leq C(U,U',F,\Lambda,\lVert \mu \rVert_{L_1^\infty},p,\alpha_0,\epsilon_0)\Bigl( \lVert \nabla_Au \rVert_{L^p(U)}+\lVert F_A \rVert_{L^p(U)}+1 \Bigr).
   \]
   The general case of $p$ follows from a bootstrap argument. We only illustrate the case for $p=2$ in what follows. Firstly, apply the above estimate for $p=4/3$, then the Sobolev embedding $L_2^{4/3} \hookrightarrow L_1^4$ implies $du\in L^4$. Therefore,
   \[
     \lVert \lvert du \rvert^2 \rVert_{L^2(U)}\leq\lVert du \rVert_{L^4(U)}^2\leq C\Bigl( \lVert \nabla_Au \rVert_{L^2(U)}+\lVert F_A \rVert_{L^2(U)}+1 \Bigr)^2.
   \]
   Since $L_1^2 \hookrightarrow L^4$ and $\lVert A \rVert_{L_1^2(U)}\leq C\lVert F_A \rVert_{L^2(U)}$ by~\eqref{item:coulomb-est} of~\autoref{thm:coulomb},
   \begin{align*}
     \lVert \lvert Adu \rvert \rVert_{L^2(U)}&\leq\lVert A \rVert_{L^4(U)}\lVert du \rVert_{L^4(U)} \leq C\Bigl(  \lVert \nabla_Au \rVert_{L^2(U)}+\lVert F_A \rVert_{L^2(U)}+1  \Bigr),
   \end{align*}
   where $C$ depends on $U,U',F,\Lambda,\lVert \mu \rVert_{L_1^\infty},\alpha_0$ and $\epsilon_0$. Now, the standard $L^2$ estimate gives similar to~\eqref{eq:eta-u-l2p-est},
   \begin{align*}
     \lVert \eta u \rVert_{L_2^2(U)}&\leq C(U,U',F,\Lambda,\lVert \mu \rVert_{L_1^\infty})
     \begin{multlined}[t]
       \Bigl( (\alpha-1)\lVert d^2(\eta u) \rVert_{L^2(U)}+\lVert \lvert du \rvert^2 \rVert_{L^2(U)}\\
         +\lVert \lvert Adu \rvert \rVert_{L^2(U)}+\lVert A \rVert_{L_1^{2}(U)}\\
       +\lVert A^2u \rVert_{L^2(U)}+\lVert \nabla_Au \rVert_{L^2(U)}+1 \Bigr),
     \end{multlined}
   \end{align*}
   and we can proceed as before to show the required estimate holds for $p=2$.
 \end{proof}
 Since the equation of the connection $A$ is subcritical in dimension 2, we can prove the following lemma.
 \begin{lem}\label{lem:eps-reg-A}
   For any $1<p<2$, there exists $\alpha_0=\alpha_0(p)>1$ such that for any $1<\alpha<\alpha_0$, if $(A,u)\in \mathscr{A}(U)\times \mathscr{S}_K(U)$ is a smooth pair which satisfies~\eqref{eq:A-u-strong-Fermi} for $1<\alpha<\alpha_0$ with $\mathcal{L}_\alpha(A,u;U)\leq\Lambda<+\infty$, then
   \[
     \lVert A \rVert_{L_2^p(U',\Omega^1(\mathfrak{g}))}\leq C\left( \lVert \nabla_Au \rVert_{L^2(U,\Omega^1(F)}+\lVert F_A \rVert_{L^2(U,\Omega^2(\mathfrak{g}))} \right),
   \]
   where $U'\subset\!\subset U$ and $C>0$ is a constant depending on $U$, $U'$, $F$, $\Lambda$, $p$.
 \end{lem}
 \begin{proof}
   Note that the equation for $A$ in~\eqref{eq:A-u-strong-Fermi} is given by,
   \[
     \begin{cases}
       \Delta A-\Psi_\alpha(A,u)=0,&x\in U\\
       A_2=0,&x\in\partial U\\
       \frac{\partial A_1}{\partial\nu}=0,&x\in\partial^0U,
     \end{cases}
   \]
   where
   \[
     \Psi_\alpha(A,u)=\left\langle dA,A \right\rangle+\left\langle A,[A,A] \right\rangle-\alpha(1+|\nabla_Au|^2)^{\alpha-1}\left\langle \nabla_Au,u \right\rangle.
   \]
   By H\"older's inequality and the Sobolev embedding, for any $1<p<2$, let $p^*=2p/(2-p)$, we have
   \[
     \lVert \left\langle dA,A \right\rangle \rVert_{L^p(U)}\leq C\lVert dA \rVert_{L^2(U)}\lVert A \rVert_{L^{p^*}(U)}\leq C\lVert A \rVert_{L_1^2(U)}^2,
   \]
   and since $L_1^2 \hookrightarrow L^q$, for any $1<q<+\infty$,
   \[
     \lVert \left\langle A,[A,A] \right\rangle \rVert_{L^p(U)}\leq C\lVert A \rVert_{L^{2p}(U)}\leq C\lVert A \rVert_{L_1^2(U)}^3.
   \]
   As we already assumed that $A$ is in Coulomb gauge, by~\autoref{thm:coulomb},
   \[
     \lVert A \rVert_{L_1^2(U)}\leq C\lVert F_A \rVert_{L^2(U)}.
   \]
   It is easy to show, for $\alpha^*$ with $\frac{1}{\alpha^*}=\frac{1}{2}+\frac{\alpha-1}{\alpha}$,
   \begin{align*}
     \left\lVert \left(1+\lvert \nabla_Au \rvert^2\right)^{\alpha-1}\left\langle \nabla_Au,u \right\rangle \right\rVert_{L^{\alpha^*}(U)}
    &\leq\left\lVert \left( 1+\lvert \nabla_Au \rvert^2 \right)^\alpha \right\rVert_{L^1(U)}^{(\alpha-1)/\alpha}\lVert \left\langle \nabla_Au,u \right\rangle \rVert_{L^2(U)}\\
    &\leq C(F) \Lambda^{(\alpha-1)/\alpha}\lVert \nabla_Au \rVert_{L^2(U)}.
   \end{align*}
   Thus, for any $1<p<2$, we can take $\alpha(p)=2p/(3p-2)\in(1,2)$, such that for any $1<\alpha\leq\alpha(p)$, we have $p\leq \alpha^*$ and
   \[
     \left\lVert \left(1+\lvert \nabla_Au \rvert^2\right)^{\alpha-1}\left\langle \nabla_Au,u \right\rangle \right\rVert_{L^p(U)}
     \leq C(F,\Lambda,p)\lVert \nabla_Au \rVert_{L^2(U)}.
   \]
   The $L^p$-estimate (see~\autoref{lem:Lp-Dirichlet-Neumann}) implies that, for any $U'\subset\!\subset U$,
   \[
     \lVert A \rVert_{L_2^p(U')}\leq C(U,U',F,\Lambda,p)\left( \lVert \nabla_Au \rVert_{L^2(U)}+\lVert F_A \rVert_{L^2(U)} \right).
   \]
 \end{proof}
 In application, we also need the scaled version of small energy estimate. For any $r$, $0<r<r_0<1$ (such that $A$ is in Coulomb gauge over $U_r$), and any fixed point $x_0\in U_r$, define the scaling map $\lambda_r \mathpunct{:}U\to U_r$, $x\mapsto x_0+rx$. If $(A,u)\in \mathscr{A}(U_r)\times \mathscr{S}_K(U_r)$ is a smooth pair which satisfies~\eqref{eq:A-u-strong-Fermi} with $\mathcal{L}_\alpha(A,u;U_r)\leq\Lambda<+\infty$, then it is easy to show, the pullback connection $\hat A \mathpunct{:}=\lambda_r^*A$ (which is in Coulomb gauge over $U$) and the pullback section $\hat u \mathpunct{:}=\lambda_r^*u=u\circ\lambda_r$ are locally given by
 \begin{equation}\label{eq:A-u-scale}
   \hat A(x) \mathpunct{:}=\lambda_r^*A(x)=rA(x_0+rx),\quad
   \hat u(x) \mathpunct{:}=\lambda_r^*u=u(x_0+r x)
 \end{equation}
 respectively. Therefore, $(\hat A,\hat u)$ satisfies
 \begin{equation}\label{eq:scale}
   \begin{cases}
     \Delta_\Sigma \hat u-2(\alpha-1)\frac{\left\langle \nabla_{\hat A}^2\hat u,\nabla_{\hat A}\hat u \right\rangle\nabla_{\hat A}\hat u}{r^2+\lvert \nabla_{\hat A}\hat u \rvert^2}-\hat\Phi_\alpha(\hat A,\hat u)=0,&x\in U\\
     \Delta\hat A-\hat\Psi_\alpha(\hat A,\hat u)=0,&x\in U\\
     \hat u^a=0,\quad a=k+1,\ldots,n,&x\in\partial^0U\\
     \frac{\partial\hat u^a}{\partial \nu}=0,\quad a=1,2,\ldots,k,&x\in\partial^0U\\
     \hat A_1=0,&x\in\partial U\\
     \frac{\partial \hat A_2}{\partial \nu}=0,&x\in\partial^0U,
   \end{cases}
 \end{equation}
 where $\partial^0U=\left\{ (x-x_0)/r:x\in\partial\Sigma\cap U_r \right\}$, $0<r<r_0<1$,
 \[
   \hat \Phi_\alpha(\hat A,\hat u)=\Gamma(\hat u)\left( d\hat u,d\hat u \right)+2\hat A\cdot d\hat u+\hat A\cdot \hat A\cdot \hat u+r^2 \frac{\nabla\mu(\hat u)(\mu(\hat u))}{\alpha(1+r^{-2}\lvert \nabla_{\hat A}\hat u \rvert^2)^{\alpha-1}}
 \]
 and
 \[
   \hat\Psi_\alpha(\hat A,\hat u)=\left\langle\hat A,d\hat A\right\rangle+\left\langle \hat A,[\hat A,\hat A] \right\rangle-\alpha r^2\left( 1+r^{-2}\lvert \nabla_{\hat A}\hat u \rvert^2 \right)^{\alpha-1}\left\langle \nabla_{\hat A}\hat u,\hat u \right\rangle.
 \]
 \begin{cor}\label{cor:eps-regularity-scaled}
   There exist $\epsilon_0>0$ and $\alpha_0>0$, such that for any smooth $(\hat A,\hat u)\in \mathscr{A}(U)\times \mathscr{S}_K(U)$ which solves~\eqref{eq:scale} for $1<\alpha<\alpha_0$, and for any $p>1$, if $\hat A$, $\hat u$ satisfies
   \begin{equation}\label{eq:scale-condi}
     \lVert \nabla_{\hat A}\hat u \rVert_{L^2(U)}\leq\eps_0,
   \end{equation}
   then for any $k=2,3,\ldots$,
   \begin{multline*}
     \lVert \hat u -\bar{\hat u}\rVert_{L_k^p(U_{1/2},F)}\\
     \leq C(\mathrm{diam}(U),F,\Lambda,\lVert \mu \rVert_{L_1^\infty},p,k,\alpha_0,\epsilon_0)\left( \lVert \nabla_{\hat A}\hat u \rVert_{L^2(U,\Omega^1(F))}+\lVert F_{\hat A} \rVert_{L^2(U,\Omega^2(\mathfrak{g}))}+1 \right),
   \end{multline*}
   and
   \[
     \lVert \hat A \rVert_{L_k^p(U_{1/2},F)}\leq C(\mathrm{diam}(U),F,\Lambda,p,k)\left( \lVert \nabla_{\hat A}\hat u \rVert_{L^2(U,\Omega^1(F))}+\lVert F_{\hat A} \rVert_{L^2(U,\Omega^2(\mathfrak{g}))} \right),
   \]
   where $\Lambda$ is the bound of $\mathcal{L}_\alpha(A,u;U_r)$.
 \end{cor}
 \begin{proof}
   Recall that harmonic maps are scaling invariant in dimension 2, although our coupled equation is not scaling invariant anymore, it behaves well under scaling. We only prove the case $1<p<2$ and $k=2$, the general case follows from bootstrap argument as illustrated at the end of~\autoref{sec:smoothness}.

   For the estimate of $\hat u$, the proof is almost the same as~\autoref{lem:eps-regulairty}. Note that
   \begin{gather*}
     \left\lvert \frac{\left\langle \nabla_{\hat A}^2\hat u,\nabla_{\hat A}\hat u \right\rangle\nabla_{\hat A}\hat u}{r^2+\lvert \nabla_{\hat A}\hat u \rvert^2} \right\rvert \leq C\lvert \nabla_{\hat A}^2\hat u \rvert,\\
     \left\lvert r^2 \frac{\nabla\mu(\hat u)(\mu(\hat u))}{\left( 1+r^{-2}\lvert \nabla_{\hat A}\hat u \rvert^2 \right)^{\alpha-1}} \right\rvert\leq C(\lVert \mu \rVert_{L_1^\infty}).
   \end{gather*}
   Multiplying the equation of $\hat u$ by the cutoff function defined in~\autoref{lem:eps-regulairty} (note that $\lvert \nabla \eta \rvert\leq C(\mathrm{diam}(U))$), it is easy to show
   \begin{align*}
     \Delta_\Sigma(\eta\hat u)&\leq C\left( \eta\lvert \Delta_\Sigma\hat u \rvert+\lvert d\hat u \rvert+\lvert \hat u \rvert \right)\\
                              &\leq C\left( \eta(\alpha-1)\lvert \nabla_{\hat A}^2\hat u \rvert+ \eta\lvert \hat\Phi_\alpha(\hat A,\hat u) \rvert\lvert du \rvert+\lvert u \rvert  \right)\\
                              &\leq C\left( (\alpha-1)\lvert d^2(\eta\hat u) \rvert+\lvert d(\eta\hat u) \rvert\lvert d\hat u \rvert+\lvert d\hat A\hat u \rvert+\lvert \hat Ad\hat u \rvert +\lvert d\hat u \rvert+\lvert \hat u \rvert +1\right)\\
                              &\leq C\left[ (\alpha-1)\lvert d^2(\eta\hat u) \rvert+\lvert d(\eta\hat u) \rvert\lvert \nabla_{\hat A}\hat u \rvert+\lvert \nabla_{\hat A}\hat u \rvert\lvert \hat A \rvert
                              +\lvert \hat A \rvert+\lvert \hat A^2 \rvert+\lvert d\hat A \rvert+\lvert d\hat u \rvert+\lvert \hat u \rvert +1 \right],
   \end{align*}
   where $C$ is a constant depending on $\mathrm{diam}(U),\Lambda,F,\lVert \mu \rVert_{L_1^\infty}$. Thus, we can control $\Delta_\Sigma(\eta \hat u)$ as in~\autoref{lem:eps-regulairty} and show the required estimate.

   Next, we prove the required estimate for $\hat A$. The proof is almost the same as~\autoref{lem:eps-reg-A}, by noting that
   \begin{align*}
     r^2\left(1+r^{-2}\lvert \nabla_{\hat A}\hat u \rvert^2\right)^{\alpha-1}\left\langle \nabla_{\hat A}\hat u,\hat u \right\rangle
  &=r^3\left( 1+\lvert \nabla_Au \rvert^2 \right)^{\alpha-1}\left\langle \nabla_Au,u \right\rangle,\\
  \lVert r^2\left(1+r^{-2}\lvert \nabla_{\hat A}\hat u \rvert^2\right)^{\alpha-1}\left\langle \nabla_{\hat A}\hat u,\hat u \right\rangle \rVert_{L^{\alpha^*}(U)}
  &\leq Cr^3\lVert \left(1+\lvert \nabla_Au \rvert^2\right)^{\alpha} \rVert_{L^1(U_r)}^{\frac{\alpha-1}{\alpha}}\lVert \nabla_Au \rVert_{L^2(U_r)}\\
  &\leq C\cdot\Lambda^{\frac{\alpha-1}{\alpha}}\cdot\lVert\nabla_{\hat A}\hat u \rVert_{L^2(U)}.
   \end{align*}
 \end{proof}
 \subsection{Removal of singularity for approximated harmonic maps}
 The following lemma is an extension of the classical removable singularity theorem for harmonic maps (see~\citelist{\cite{SacksUhlenbeck1981existence}*{Thm.~3.6}\cite{Fraser2000free}*{Thm.~1.10}}), which will be applied to the weak limit in the blow-up process to show that the isolated singularities are all removable. The proof given here is based on the regularity theorem of weak solution instead, comparing to the classical method involving energy decay estimates \citelist{\cite{SacksUhlenbeck1981existence}\cite{Fraser2000free}}. Here we only state the boundary version, the interior case can be found in~\cite{Song2011CriticalYangMillsHiggs}*{Thm.~4.3}.
 \begin{lem}[\cite{JostLiuZhu2016qualitative}*{Thm.~3.6}]\label{lem:removable-singularity}
   Suppose $u$ is a $L_{2,\mathrm{loc}}^2$-map from a neighborhood $U^\circ \mathpunct{:}=U\setminus\left\{ 0 \right\}$ of $0\in\partial^0U$ to $F$ with finite Dirichlet energy and satisfies the following equation weakly under Fermi coordinates (see also~\eqref{eq:loc})
   \begin{equation}\label{eq:u-removable-singularity}
     \begin{cases}
       \Delta_\Sigma u-\Gamma(u)(du,du)=f\in L^p(U^\circ),&x\in U^\circ\\
       \frac{\partial u^a}{\partial \nu}=0,\quad a=1,2,\ldots,k,&x\in\partial\Sigma\cap U^\circ\\
       u^a=0,\quad a=k+1,\ldots,n,&x\in\partial\Sigma\cap U^\circ,
     \end{cases}
   \end{equation}
   %If $A\in L_2^2(U^\circ)$,
   for some $p\geq2$. Then $u$ can be extended to a $L_{2,\mathrm{loc}}^p$-map over $U$ and it preserves the free boundary condition.
 \end{lem}

 \section{Convergence and blow-up}\label{sec:blowup}
 The following bubbling convergence argument is almost standard, the main difference is the possible phenomenon of boundary blow-ups.

 \begin{proof}[Proof of~\autoref{thm:blowup}]
   Let $\alpha_0$ be the same constant in \autoref{lem:eps-regulairty}. Suppose $\left\{ x_1,\ldots,x_L \right\}\subset \mathcal{S}$. By the definition of $\mathcal{S}$, for $r>0$ small enough such that $\left\{ U_r(x_j) \right\}_{j=1}^L$ are mutually disjoint and for all but finite many $\left\{ \alpha \right\}$, we have
   \[
     \int_{U_r(x_j)}\lvert \nabla_{A_{\alpha}}\phi_{\alpha} \rvert^2\geq\epsilon_0/2.
   \]
   Summing over $j=1,\ldots,L$, we see that
   \[
     +\infty>\Lambda\geq \mathcal{L}_{\alpha}(A_{\alpha},\phi_{\alpha})\geq\sum_{j=1}^L\int_{U_{r}(x_j)}\lvert \nabla_{A_{\alpha}}\phi_{\alpha} \rvert^2\geq L\epsilon_0/2,
   \]
   which clearly implies the finiteness of $\mathcal{S}$.

   To show the strong convergence over regular points in $\Sigma\setminus \mathcal{S}$, we note first that, by the remark after~\autoref{lem:eps-regulairty}, there exists $r_0>0$ independent of $\alpha$, such that $A_\alpha$ is in Coulomb gauge over $U_r$, provided that $r\leq r_0$.
   Then~\autoref{lem:eps-reg-A} implies, for $1<p<2$, there exists $\alpha(p)>1$, such that for any $1<\alpha<\alpha(p)$, $\lVert A_\alpha \rVert_{L_2^p(U_r,\Omega^1(\mathfrak{g}))}$ are uniformly bounded. Next, we show the $C^0$ convergence of $A_\alpha\to A_\infty$. For that purpose, covering $\Sigma$ with discs or half-discs with radius less than $r_0/2$, denote them by $\left\{ U_{i} \right\}$. The above discussion shows that under some local trivialization $\left\{ \sigma_{\alpha,i}\mathpunct{:}\mathcal{P}|_{U_i}\to U_i\times G \right\}$, if we write $A_\alpha$ locally as $d+A_{\alpha,i}$, then for any $U_i$, any $1<p<2$ and any $1<\alpha<\alpha(p)$,
   \begin{equation}\label{eq:2p-bdd-A}
     \lVert A_{\alpha,i} \rVert_{L_2^p(U_i)}\leq C(F,\Lambda,p).
   \end{equation}
   Thus, we can assume that $A_{\alpha,i}\to A_i$ weakly in $L_2^p(U_i)$ and strongly in $C^0(U_i)$ as $\alpha\to1$.
   \begin{claim}
     $\left\{ A_i \right\}$ represents a $L_2^p$ connection $A_\infty$ on $\mathcal{P}$, i.e., $A_i\in L_2^p(U_i,\Omega^1(\mathfrak{g}))$ and there exist transition functions $\left\{ \tau_{ij}\in L_3^p(U_{ij},G) \right\}$, $U_{ij}\mathpunct{:}=U_i\cap U_j$, such that
     \[
       A_j=\tau_{ij}^*A_i=\tau^{-1}_{ij}d\tau_{ij}+\tau^{-1}_{ij}A_i\tau_{ij}.
     \]
   \end{claim}
   In fact, on any $U_{ij}\neq\emptyset$, we have transition functions $\left\{ \tau_{\alpha,ij}\mathpunct{:}U_{ij}\to G \right\}$ such that $\pi_2\circ\sigma_{\alpha,i}=\tau_{\alpha,ij}\circ\pi_2\circ\sigma_{\alpha,j}$, where $\pi_2$ is the projection to the second component. $\left\{ A_{\alpha,i} \right\}$ transform as
   \[
     A_{\alpha,j}=\tau_{\alpha,ij}^{-1}d\tau_{\alpha,ij}+\tau_{\alpha,ij}^{-1}A_{\alpha,i}\tau_{\alpha,ij}\Longleftrightarrow
     d\tau_{\alpha,ij}=\tau_{\alpha,ij}A_{\alpha,j}-A_{\alpha,i}\tau_{\alpha,ij}.
   \]
   Since $G$ is compact,~\eqref{eq:2p-bdd-A} and the above relation imply that, for $p^*=2p/(2-p)$,
   \begin{align*}
     \lVert d\tau_{\alpha,ij} \rVert_{L^{p^*}(U_{ij})}&\leq C(G)\left( \lVert A_{\alpha,i} \rVert_{L^{p^*}(U_i)}+\lVert A_{\alpha,j} \rVert_{L^{p^*}(U_j)} \right)\\
                                                      &\leq C(G)\left( \lVert A_{\alpha,i} \rVert_{L_1^p(U_i)} + \lVert A_{\alpha,j} \rVert_{L_1^p(U_j)}  \right)\\
                                                      &\leq C(G,F,\Lambda,p).
   \end{align*}
   Since $\tau_{\alpha,ij}\in L^\infty(U_{ij})$ (because $G$ is compact), by employing the Sobolev multiplication theorems $L_2^p\times L_1^{p^*}\to L_1^p$ and $L_2^p\times L_2^p\to L_2^p$ (see~\cite{Wehrheim2004Uhlenbeck}*{Lem.~B.3}), we obtain the $L_3^p(U_{ij})$-uniform boundedness of $\left\{ \tau_{\alpha,ij} \right\}$. By weak compactness, we may assume that $\tau_{\alpha,ij}$ converges to some $\tau_{ij}$ weakly in $L_3^p(U_{ij})$ and strongly in $C^0(U_{ij})$ as $\alpha\to1$. It is clear that the co-cycle condition $\tau_{ik}=\tau_{ij}\circ\tau_{jk}$ is preserved and hence $\left\{ \tau_{ij} \right\}$ defines a bundle isomorphic to $\mathcal{P}$. Moreover, the relation is preserved under weak limits,
   \[
     d\tau_{ij}=\tau_{ij}A_j-A_i\tau_{ij}\Longleftrightarrow A_j=\tau_{ij}^*A_i.
   \]
   Thus, $\left\{ A_i \right\}$ represents a connection $A_\infty\in \mathscr{A}_2^p$ on $\mathcal{P}$. This finishes the proof of the claim.

   We should remark that the local convergence $A_{\alpha,i}\to A_i$ in $C^0(U_i)$ depends on the choice of trivialization $\sigma_{\alpha,i}$ and we cannot assert $A_\alpha\to A_\infty$ in $C^0$ directly. But we can apply the patching argument similar to the weak compactness of Yang--Mills connections (see~\cite{Uhlenbeck1982Connections}*{Thm.~3.6}) to show that there exist gauge transformations $\left\{ S_\alpha \right\}\subset \mathscr{G}_3^p$, such that $S_\alpha^*A_\alpha\to A_\infty$ strongly in $C^0$ sense. That is, $A_\alpha\to A_\infty$ in $C^0(\Sigma)$ modulo gauge. Since $\mathcal{L}(A_\alpha,\phi_\alpha)$ is gauge invariant, we will identify $S_\alpha^*A_\alpha$ and $S_\alpha^*\phi_\alpha$ with $A_\alpha$ and $\phi_\alpha$ hereafter.

   To show the strong convergence of sections $\left\{ \phi_\alpha \right\}$ over $\Sigma\setminus \mathcal{S}$, we note first that,  the above argument can be started with any cover with radii are less than $r_0/2$. Now, by the definition of regular set, for any $x\in\Sigma\setminus \mathcal{S}$, there exist $r^0\in(0,r_0)$ and $\alpha^0\in(0,\alpha_0)$, such that for any $U_r(x)\subset \Sigma$ centered at $x$ with radius $r\leq r^0$, we have
   \[
     \int_{U_r(x)}\lvert \nabla_{A_\alpha}\phi_\alpha \rvert^2<\epsilon_0,\quad\forall 1<\alpha\leq\alpha^0.
   \]
   Since by the choice of $U_r(x)$, $r<r^0<r_0$, we can assume that $A_{\alpha}$ is in Coulomb gauge with estimate~\eqref{eq:2p-bdd-A} over $U\subset U_r(x)$. If we denote the corresponding local trivialization by $\sigma_{\alpha}$ and write
   \[
     \phi_{\alpha}(x)=\sigma_{\alpha}\circ\phi_\alpha(x)=(x,u_{\alpha}(x)),
   \]
   then~\autoref{lem:eps-regulairty} implies that, for any $1<p\leq 2$,
   \[
     \lVert u_{\alpha} \rVert_{L_2^p(U')}\leq C(U, U',F,\Lambda,\lVert \mu \rVert_{L_1^\infty(F)},p,\alpha_0,\epsilon_0),\quad U'\subset\!\subset U.
   \]
   Since $(A_{\alpha},u_{\alpha})$ satisfies~\eqref{eq:A-u-strong-Fermi}, we can bootstrap the regularity as in the proof of smoothness of critical points of $\alpha$-YMH functional (see the end of~\autoref{sec:smoothness}) and conclude that $\left\{ A_\alpha \right\}$ converges to $A_\infty$ in $C^\infty(U')$ and $\left\{ u_{\alpha} \right\}$ converges to some $u$ in $C^\infty(U')$ as $\alpha\to1$. By the arbitrariness of $x\in \Sigma\setminus \mathcal{S}$, we can construct a cover $\left\{ U_i' \right\}$ of $\Sigma\setminus \mathcal{S}$ and local trivializations $\sigma_{\alpha,i}$, such that for any $i$, $A_{\alpha,i}\to A_i$ in $C^\infty(U_i')$ and $u_{\alpha,i}\to u_i$ in $C^\infty(U_i')$ as $\alpha\to1$. Since the consistence condition $ u_j=\tau_{ij}u_i $ is preserved on each $U_i'\cap U_j'$, $\left\{ u_i \right\}$ represents a section $\phi_\infty\in \mathscr{S}_K$ over $\Sigma\setminus \mathcal{S}$. A patching argument as before shows that, there exists some $S_\alpha\in C_{\mathrm{loc}}^\infty(\Sigma\setminus \mathcal{S})$, such that $S_\alpha^*\phi_\alpha\to \phi_\infty$ and $S_\alpha^*A_\alpha\to A_\infty$ in $C_{\mathrm{loc}}^\infty(\Sigma\setminus \mathcal{S})$. Clearly, by taking $\alpha\to1$ in~\eqref{eq:A-u-strong-Fermi}, $\phi_\infty$ satisfies the first equation of~\eqref{eq:loc} locally and the corresponding boundary condition over some neighborhood $U\setminus\left\{ x \right\}$, $x\in \mathcal{S}$. The removal of regularity theorem (see~\autoref{lem:removable-singularity}) asserts that $\phi_\infty$ extends to a smooth section over $\Sigma$ and we finish the first part of the theorem.

   To show the second part, without loss of generality, suppose that $x_0=0\in \mathcal{S}\cap\partial\Sigma$ and $U=U_1$ (unit half disc) is a neighborhood of origin such that it is the unique isolated singularity in $U$. Let $\sigma_\alpha$ be a local trivialization over $U$ and $\left\{ u_\alpha \right\}$ be the local representation of $\left\{ \phi_\alpha \right\}$ as before. Without loss of generality, we may assume that the radius of $U$ is less than $r_0$, such that $A_\alpha$ is in Coulomb gauge. Set
   \[
     1/r_\alpha=\max_{U}\lvert \nabla_{A_\alpha}u_\alpha \rvert=\lvert \nabla_{A_\alpha}u_\alpha \rvert(x_\alpha),
   \]
   and let $\lambda_{r_\alpha} \mathpunct{:}x\mapsto x_\alpha+r_\alpha x$ be the scaling mapping. We already shown that the pullback connection and pullback section are locally given by (see~\eqref{eq:A-u-scale})
   \begin{align*}
     \hat A_\alpha(x)&\mathpunct{:}=\lambda_{r_\alpha}^*A_\alpha(x)=r_\alpha A_\alpha(x_\alpha+r_\alpha x),\\
     \hat u_\alpha(x)&\mathpunct{:}=\lambda_{r_\alpha}^*u_\alpha(x)=u_\alpha\circ \lambda_{r_\alpha}(x)=u_\alpha(x_\alpha+r_\alpha x).
   \end{align*}
   The following blow-up argument is standard, and we summarize it in the following claim as a complement.
   \begin{claim}
     With the above notations and assumptions, we have
     \begin{enumerate}
       \item $r_\alpha\to0$ as $\alpha\to1$;\label{item:blowup:a}
       \item $x_\alpha\to 0$ as $\alpha\to1$;\label{item:blowup:b}
       \item Define $(\hat A_\alpha, \hat u_\alpha)$ as above, then there are two cases, where harmonic spheres and harmonic discs split off respectively.\label{item:blowup:c}
         \begin{itemize}
           \item Harmonic spheres: $\mathrm{dist}(x_\alpha,U\cap\partial\Sigma)/r_\alpha\to\infty$;
           \item Harmonic discs: $\mathrm{dist}(x_\alpha,U\cap\partial\Sigma)/r_\alpha\to\rho<+\infty$.
         \end{itemize}
     \end{enumerate}
   \end{claim}
   If~\eqref{item:blowup:a} is not true, then $\lVert \nabla_{A_\alpha}u_\alpha \rVert_{L^\infty(U)}$ are uniformly bounded. This contradicts the fact that $x_0=0$ is a singularity of $\left\{ u_\alpha \right\}$ in $U$.

   For~\eqref{item:blowup:b}, suppose that $x_\alpha\to x^0\neq 0$ as $\alpha\to1$, then since $x^0$ is a regular point, there exist $\delta>0$ and $\alpha^0\in(1,\alpha_0)$, such that we can apply~\autoref{lem:eps-regulairty} and~\autoref{lem:eps-reg-A} to show that,
   \[
     \frac{1}{r_\alpha}=\lvert \nabla_{A_\alpha}u_\alpha \rvert(x_\alpha)\leq\lVert \nabla_{A_\alpha}u_\alpha \rVert_{L^\infty(U_\delta(x^0))}\leq C<+\infty,
   \]
   take $\alpha\to1$ we see that it contradicts ~\eqref{item:blowup:a}.

   For~\eqref{item:blowup:c}, we only show the case of splitting-off of harmonic discs with free boundary, because the harmonic sphere case can be derived in a very similar way. Firstly, we can take a proper coordinate system with origin at $x_0=0$ and $x_1$-axis pointing to the interior of $\Sigma$, $x_2$-axis tangent to $\partial\Sigma$ at $0$. The scaled maps $(\hat A_\alpha,\hat u_\alpha)$ satisfy~\eqref{eq:scale} with $r$ replaced by $r_\alpha$. Since $\nabla_{\hat A_\alpha}\hat u_\alpha(x)=r_\alpha\nabla_{A_\alpha}u_\alpha(x_\alpha+r_\alpha x)$,
   \[
     \lVert \nabla_{\hat A_\alpha}\hat u_\alpha \rVert_{L^\infty(U_{1/r_\alpha})}=r_\alpha\lVert \nabla_{A_\alpha}u_\alpha \rVert_{L^\infty(U)}=1,
   \]
   by the choice of $r_\alpha$, we can apply~\autoref{cor:eps-regularity-scaled} on each $\tilde{U}\subset U_{1/(2r_\alpha)}$ to $(\hat A,\hat u_\alpha)$ and show that, for $k=1,2,\ldots$,
   \begin{align*}
     \lVert \hat u_\alpha \rVert_{C^k(U_{1/(2r_\alpha)})}&\leq C(\mathrm{diam}(U),F,\Lambda,\lVert \mu \rVert_{L_1^\infty(F)},\alpha_0,\epsilon_0,k),\\
     \lVert \hat A_\alpha \rVert_{C^k(U_{1/(2r_\alpha)})}&\leq C(\mathrm{diam}(U),F,\Lambda,k).
   \end{align*}
   Moreover, since $\hat A_\alpha$ is in Coulomb gauge, by \autoref{thm:coulomb},
   \[
     \lVert \hat A_\alpha \rVert_{L_1^2(U_{1/(2r_\alpha)})}\leq C\lVert F_{\hat A_\alpha} \rVert_{L^2(U_{1/(2r_\alpha)})}=Cr_\alpha\lVert F_{A_\alpha} \rVert_{L^2(U_{1/2})}\to0.
   \]
   Therefore, we obtain the following strong convergence in $C_{\mathrm{loc}}^\infty(\mathbb{R}^2_{\rho,+})$, where the right half plane $\mathbb{R}_{\rho,+}^2 \mathpunct{:}=\left\{ x=(x_1,x_2):x_1>-\rho \right\}$,
   \[
     \hat A_\alpha\to 0,\quad\hat u_\alpha\to w.
   \]
   Clearly, the equation of $w$ is given by
   \[
     \begin{cases}
       \Delta_\Sigma w-\Gamma(w)(dw,dw)=0,&x\in \mathbb{R}^2_{\rho,+}\\
       \frac{\partial w^a}{\partial x_1}=0,\quad a=1,2,\ldots,k,&x_1=-\rho\\
       w^a=0,\quad a=k+1,\ldots,n,&x_1=-\rho.
     \end{cases}
   \]
   By the removal of singularity theorem for harmonic maps (see~\autoref{lem:removable-singularity}) and the conformal invariance of $w$, $w$ extends to a harmonic map on the disc $B$ with free boundary $w(\partial B)\subset K$. Therefore, at each singularity $x_0\in \mathcal{S}\cap\partial\Sigma$, we obtain a harmonic disc or a harmonic sphere, which is called a bubble. This finishes the proof of \autoref{thm:blowup}.
 \end{proof}
 \appendix
 \section{Some regularity results and estimates}\label{sec:bdry-regularity}
 It is well-known that for a weakly harmonic map $u$, the equation of $u$ has anti-symmetric structure $\Omega$ with $\lVert \Omega \rVert_{L^2}\leq C\lVert \nabla u \rVert_{L^2}$ and the following regularity and estimate hold.
 \begin{lem}[see~\cite{SharpZhu2016Regularity}*{Thm.~1.2}]\label{lem:Lp-Dirichlet-Neumann}
   Suppose $u\in L_1^2(D_1,\mathbb{R}^n)$ is a weak solution of
   \[
     \begin{cases}
       \Delta u+\Omega\cdot \nabla u=f\in L^p(D_1,\mathbb{R}^n),&x\in D_1\\
       \frac{\partial u^a}{\partial\nu}=g^a\in L_{1,\partial}^p(\partial^0D_1,\mathbb{R}^n),&x\in\partial^0D_1,\quad 1\leq a\leq k\\
       u^a=h^a\in L_{2,\partial}^p(\partial^0D_1,\mathbb{R}^n),&x\in\partial^0D_1,\quad k+1\leq a\leq n,
     \end{cases}
   \]
   where $\Omega\in L^2(D_1,\mathfrak{so}(n)\times\wedge^1\mathbb{R}^2)$, $1<p<2$ and boundary Sobolev space is defined as
   \[
     L_{k,\partial}^p(\partial^0D_1)\mathpunct{:}=\left\{ f\in L^1(\partial^0D_1):f=\tilde{f}|_{\partial^0D_1},\,\tilde{f}\in L_k^p(D_1) \right\}
   \]
   with norm
   \[
     \lVert f \rVert_{L_{k,\partial}^p(\partial^0D_1)}\mathpunct{:}=\inf_{\tilde{f}\in L_k^p(D_1),\tilde{f}|_{\partial^0D_1}=f}\lVert \tilde{f} \rVert_{L_k^p(D_1)}.
   \]
   Then, $u\in L_2^p(\overline{D_{1/2}},\mathbb{R}^n)$ and
   \[
     \lVert u \rVert_{L_2^p(D_{1/2},\mathbb{R}^n)}\leq C\left( \lVert f \rVert_{L^p(D_1,\mathbb{R}^n)}+\lVert g \rVert_{L_{1,\partial}^p(\partial^0D_1,\mathbb{R}^n)}+ \lVert h \rVert_{L_{2,\partial}^p(\partial^0D_1,\mathbb{R}^n)}+\lVert u \rVert_{L^1(D_1,\mathbb{R}^n)}\right),
   \]
   provided that $\lVert \Omega \rVert_{L^2(D_1)}\leq\eta_0=\eta_0(p,n)$.
 \end{lem}
 %\nocite{*}
 %\bibliography{free-bdry}
 %\input{\jobname.bbl}

   \end{document}